\documentclass{amsart}

\usepackage[stable]{footmisc}
\usepackage{tikz-cd,amssymb,amsmath,amsthm,mathrsfs,mathscinet,anyfontsize,latexsym,graphics,mathdots,hyperref,bm}
\usepackage[all]{xy}

\newcommand{\grdim}{\operatorname{grdim}}
\newcommand{\codim}{\operatorname{codim}}
\newcommand{\cvr}{\operatorname{-cvr}}
\newcommand{\nilp}{\operatorname{nilp}}
\newcommand{\tors}{\mathscr{T}}
\newcommand{\torf}{\mathscr{F}}
\newcommand{\Tors}{\mathfrak{T}}
\newcommand{\Torf}{\mathfrak{F}}
\newcommand{\Vgr}{\mathcal{V}}
\newcommand{\argmin}{\operatorname{argmin}}
\newcommand{\dat}{\mathfrak{d}}
\newcommand{\linT}{\mathcal{T}}
\newcommand{\sprod}[2]{\left\langle{#1},{#2}\right\rangle}
\newcommand{\Mult}{\mathfrak{M}}
\newcommand{\Exts}{\mathcal{E}}
\renewcommand{\subset}{\subseteq}
\renewcommand{\supset}{\supseteq}
\newcommand{\cn}{\mu}
\newcommand{\cpp}{\smile}
\newcommand{\zerocap}{\mathcal{Q}}
\newcommand{\prm}{\lambda}
\newcommand{\Gm}{\mathbb{G}_m}
\newcommand{\Prj}{\mathbb{P}}
\newcommand{\op}{\circ}
\newcommand{\Der}{\operatorname{Der}}
\newcommand{\Fac}{\operatorname{Fac}}
\newcommand{\GL}{\operatorname{GL}}
\newcommand{\Cat}{\mathcal{C}}
\newcommand{\Irrcomp}{\operatorname{Comp}}
\newcommand{\bfd}{\bm{d}}
\newcommand{\bfe}{\bm{e}}
\newcommand{\m}{\mathfrak{m}}
\newcommand{\n}{\mathfrak{n}}

\newcommand{\N}{\mathbb{N}}
\newcommand{\orb}{\mathcal{O}}
\newcommand{\Hom}{\operatorname{Hom}}
\newcommand{\Lin}{\operatorname{Lin}}
\newcommand{\LLin}{\mathcal{L}}
\newcommand{\Ext}{\operatorname{Ext}}
\newcommand{\ext}{\operatorname{ext}}
\newcommand{\Mat}{\operatorname{Mat}}
\renewcommand{\Im}{\operatorname{Im}}
\newcommand{\End}{\operatorname{End}}
\newcommand{\Ker}{\operatorname{Ker}}
\newcommand{\Coker}{\operatorname{Coker}}
\newcommand{\Aut}{\operatorname{Aut}}
\newcommand{\sm}[4]{{\bigl(\begin{smallmatrix}{#1}&{#2}\\{#3}&{#4}\end{smallmatrix}\bigr)}}

\newtheorem{theorem}{Theorem}[section]
\newtheorem{lemma}[theorem]{Lemma}
\newtheorem{definition}[theorem]{Definition}
\newtheorem{proposition}[theorem]{Proposition}
\newtheorem{corollary}[theorem]{Corollary}
\theoremstyle{remark}
\newtheorem{remark}[theorem]{Remark}
\newtheorem{example}[theorem]{Example}

\numberwithin{equation}{section}

\begin{document}

\title[Irreducible components of nilpotent varieties {I}]{A binary operation on irreducible components of
Lusztig's nilpotent varieties {I}: definition and properties}

\author{Avraham Aizenbud}
\email{aizenr@gmail.com}
\address{Department of Mathematics\\Weizmann Institute of Science\\Rehovot 7610001 Israel}
\author{Erez Lapid}
\email{erez.m.lapid@gmail.com}
\address{Department of Mathematics\\Weizmann Institute of Science\\Rehovot 7610001 Israel}
\date{\today}
\maketitle

\begin{abstract}
We define a binary operation on the set of irreducible components of Lusztig's nilpotent varieties of a quiver.
We study commutativity, cancellativity and associativity of this operation.
We focus on rigid irreducible components and discuss inductive ways to construct them.
\end{abstract}

\setcounter{tocdepth}{1}
\tableofcontents

\section{Introduction}

In this paper we define and study a binary operation on the set $\Irrcomp$ of irreducible components of Lusztig's nilpotent varieties
(of all graded dimensions) pertaining to a finite quiver $Q$ without loops.
Recall that $\Irrcomp$ parameterizes Lusztig's canonical basis of the negative part $U^-$ of the quantum enveloping algebra pertaining to $Q$,
as well as the vertices of the corresponding crystal graph $B(\infty)$ \cite{MR1088333, MR1182165, MR1115118, MR1458969}.
In fact, the binary operation extends the basic crystal operators.
In the case where the underlying graph of $Q$ is a simply laced Dynkin diagram, $\Irrcomp$ is in bijection with
the equivalence classes of finite-dimensional representations of $Q$.
By Gabriel's theorem, the indecomposable representations of $Q$ determine the positive roots in the root system defined by the Dynkin diagram \cite{MR0332887, MR0393065}.
Hence, in this case, $\Irrcomp$ is in bijection with multisets of positive roots.

The nilpotent varieties classify nilpotent modules of a given graded dimension of the preprojective algebra $\Pi$ of $Q$.
(The nilpotency condition is redundant in the Dynkin case.)
A special role is played by the rigid irreducible components, i.e., those containing an open orbit under the natural action
of the group of grading preserving linear isomorphisms.
Their analogues, the so-called ``real'' simple modules of either the quiver Hecke algebras (aka.\ KLR algebras)
or quantum groups, were studied by many people, including Geiss, Hernandez, Kang, Kashiwara, Kim, Leclerc, Oh and Schr\"oer,
in the context of monoidal categorification of cluster algebras \cite{MR2682185, MR2242628, MR3758148}.
We refer the reader to \cite{MR3966729} and \cite{1902.01432} for recent surveys on this topic.

The binary operation on $\Irrcomp$ is modeled after a construction of Lusztig.
Given $C_1,C_2\in\Irrcomp$ we define $C_1*C_2$ to be the Zariski closure of the constructible set formed by all
possible extensions of $x_1$ by $x_2$ for all $(x_1,x_2)$ in a suitable open subset of $C_1\times C_2$.
Our first main result (Theorem \ref{thm: app}) is that $C_1*C_2$ is indeed an irreducible component.
In the case where $C_1$ or $C_2$ corresponds to a simple nilpotent module (i.e., the graded dimension is the indicator function of a vertex $i$ of $Q$)
this gives rise to the crystal operator $\tilde f_i$ of $B(\infty)$.
Crawley-Boevey and Schr\"oer studied this construction for the module varieties of a general finitely generated algebra, and showed that
it gives rise to an irreducible subvariety, which under certain circumstances is an irreducible component \cite{MR1944812}.
More recently, in the case of nilpotent varieties, Baumann, Kamnitzer and Tingley treated cases
of ``torsion pairs'' \cite{MR3270589}.\footnote{We are grateful to the referee for pointing out this reference to us.}
However, the fact that for nilpotent varieties the result is always an irreducible component seems to have gone unnoticed until now.
This operation is neither commutative nor associative in general. However, we can analyze the lack
of commutativity and associativity. We also study the rigidity of $C_1*C_2$ in terms of conditions on $C_1$ and $C_2$.

It is known that $U^-$ is categorified by the KLR algebra $R$.
(See e.g.\ \cite[Theorem 2.1.2]{MR3758148} and the references therein for a precise statement.
We will freely use the terminology of [ibid.] in the following discussion.)
Moreover, if $R$ is symmetric, then in this categorification, at least in characteristic $0$, the dual canonical basis corresponds
to the self-dual simple $R$-modules \cite{MR2837011} (see also \cite[Theorem 2.1.4]{MR3758148}).
Denote by $C\mapsto\pi(C)$ the resulting bijection between $\Irrcomp$ and the set of equivalence classes
of simple self-dual $R$-modules.
It is conjectured that if $C$ is rigid, then $\pi(C)$ is real. (See \cite[Conjecture 18.1]{MR2822235}
for a related conjecture.)
Assuming this, it follows from \cite{MR3314831} that if $C_1$ or $C_2$ is rigid, then the socle of
$\pi(C_1)\circ\pi(C_2)$ is simple.
Here, $\circ$ denotes the convolution product defined in \cite[\S2.1]{MR3758148}.
It is then natural to conjecture that this socle is $\pi(C_1*C_2)$.

In the case $Q=A_n$, one can reformulate the conjecture in terms
of representation theory of the general linear groups over a local, non-archimedean field.
This will be done in the second part of the paper, together with a verification of the conjecture in a special case.

We now describe the contents of the paper in more detail.

In \S\ref{sec: prep} we recall the well known algebraic/geometric objects
pertaining to quivers and their representation, namely preprojective algebras and the nilpotent varieties
defined by Lusztig.
As already mentioned, the irreducible components of the nilpotent varieties are especially important.
We recall the formula of Crawley-Boevey for the dimension of $\Ext^1$ of two modules of the preprojective algebra
\cite{MR1781930} and the switch-duality of $\Ext^1$ \cite{MR2360317}.
The case where the underlying graph of the quiver is a simply laced Dynkin diagrams is discussed in detail.

In \S\ref{sec: binary} we introduce the binary operation $*$ on $\Irrcomp$, which is the main object of the paper.
The fact that it is well defined relies on a simple dimension counting argument using Crawley-Boevey's formula and
the fact that the nilpotent varieties are of pure dimension.

The central notion of rigid modules and irreducible components is recalled in \S\ref{sec: rigid}.

Cancellativity of rigid irreducible components with respect to $*$ and a divisibility criterion
are the subject matter of \S\ref{sec: cancel}.

Commutativity and associativity of $*$ (or lack thereof) are discussed in \S\ref{sec: comm} and \S\ref{sec: assoc},
respectively.

Further technical results are presented in \S\ref{sec: further}. They will be used in the second part
of the paper.

Finally, in \S\ref{sec: taudata} we go back to the description of modules of the preprojective algebra.
Roughly speaking, a $\Pi$-module can be viewed as a representation of $Q$ with supplementary data.
Following Ringel \cite{MR1648647}, in the case where the quiver does not admit cycles (which includes the Dynkin case)
this data can be formulated in terms of the Coxeter functors of Bernstein--Gel'fand--Ponomarev \cite{MR0393065}
using their cohomological properties \cite{MR607140}.
As a complement to \cite[\S8]{MR2360317} we give a switch-dual short exact sequence for $\Ext^1_\Pi$
which is useful in computations.

\subsection*{Acknowledgement}
We are grateful to Bernard Leclerc for useful correspondence.
The second-named author would like to thank the Hausdorff Research Institute for Mathematics in Bonn for its
generous hospitality in July 2021.
Conversations with Jan Schr\"oer during that time were particularly helpful and we are very grateful to him.
Last but not least, we are greatly indebted to the referee for reading the paper carefully, pointing out several
inaccuracies in the original version, and making numerous suggestions,
including a simplification of the proof of Theorem \ref{thm: app}, which also removed the characteristic $0$ assumption
made in the first version of the paper, as well as streamlining the discussion of \S\ref{sec: taudata}.

\section{Quivers, preprojective algebras and nilpotent varieties} \label{sec: prep}

Throughout this paper we fix an algebraically closed field $K$.
All pertinent objects (vector spaces, algebras, tensor products, varieties, morphisms, dimensions, etc.) are implicitly over $K$.
By convention, unless indicated otherwise, all modules are left-modules and all ideals are two-sided.
We denote the dual of a vector space $V$ by $V^*$.
For vector spaces $U,V$ we denote by $\Lin(U,V)$ the vector space of linear maps from $U$ to $V$.
If $U=V$, we write for brevity $\Lin(V)=\Lin(V,V)$.

For any finite set $A$, let $\N A$ be the free abelian monoid generated by $A$.
We think of an element in $\N A$ as either a set of elements of $A$ with multiplicities (i.e., a multiset),
or a function from $A$ to $\N$.

In this section we recall some standard definitions and facts about quivers, their representations
and related cohomological and algebraic geometric aspects.
We mainly follow Lusztig \cite{MR1088333} and Ringel \cite{MR1648647}.

Let $Q=(I,\Omega)$ be a finite quiver without loops with vertex set $I$.
The data is encoded by a function $\Omega\rightarrow I\times I$, $h\mapsto (h',h'')$
(representing arrows $h'\xrightarrow{h}h''$) such that $h'\ne h''$ for all $h\in\Omega$.
Let $KQ$ be the path algebra of $Q$. Recall that it is hereditary.
For each $i\in I$, let $e_i$ be the idempotent of $KQ$ corresponding to the trivial path with vertex $i$.
A representation of $Q$ is an $I$-graded vector space
$V=\oplus_{i\in I}V_i$ together with linear maps $A_h:V_{h'}\rightarrow V_{h''}$ for each $h\in\Omega$.
This is the same as a $KQ$-module such that $e_iV=V_i$ for all $i\in I$.
Henceforth, unless otherwise mentioned, we will only consider finite-dimensional representations.

We write $\Vgr$ for the category of finite-dimensional $I$-graded vector spaces and for simplicity write
$V\in\Vgr$ to mean that $V$ is an object in $\Vgr$. If $V$ has graded dimension $\grdim V=\bfd\in\N I$,
then we write $V\in\Vgr(\bfd)$.

Let $V\in\Vgr(\bfd)$. The representations of $Q$ on $V$ form an affine space $R_Q(V)$ of dimension
\[
\sum_{h\in\Omega}\bfd(h')\bfd(h'').
\]
The group $G_V=\prod_{i\in I}\GL(V_i)$ of grading-preserving automorphisms of $V$ acts linearly on $R_Q(V)$ by conjugation.
The orbits are the isomorphism classes of representations of $Q$ of graded dimension $\bfd$.

It is easy to write a projective resolution of $KQ$ as a bi-module over itself.
Consequently, for any $Q$-representations $M$ and $N$ with data $\mu_h:M_{h'}\rightarrow M_{h''}$ and $\nu_h:N_{h'}\rightarrow N_{h''}$,
$h\in\Omega$, we have an exact sequence
\begin{multline} \label{eq: ESQ}
0\rightarrow\Hom_Q(M,N)\rightarrow\oplus_{i\in I}\Lin(M_i,N_i)\xrightarrow{\alpha_{\mu,\nu}}
\oplus_{h\in\Omega}\Lin(M_{h'},N_{h''})\rightarrow\\\Ext^1_Q(M,N)\rightarrow0
\end{multline}
where the middle morphism is defined by
\begin{equation} \label{def: alphaMN}
\alpha_{\mu,\nu}((A_i)_{i\in I})=(\nu_hA_{h'}-A_{h''}\mu_h)_{h\in\Omega}.
\end{equation}
In particular, if $M$ and $N$ are of graded dimension $\bfd$ and $\bfe$ respectively, then
\begin{multline} \label{eq: EulerQ}
\dim\Hom_Q(M,N)-\dim\Ext^1_Q(M,N)=\sprod{\bfd}{\bfe}_Q\\:=\sum_{i\in I}\bfd(i)\bfe(i)-\sum_{h\in\Omega}\bfd(h')\bfe(h'').
\end{multline}
Note that if $\grdim V=\bfd$, then $\sprod{\bfd}{\bfd}_Q=\dim G_V-\dim R_Q(V)$.

For every $i\in I$ denote by $S_Q(i)$ the (simple) representation whose graded dimension is the indicator
function of $i$. It is unique up to isomorphism. Moreover, if $\bfd$ is concentrated at $i$,
i.e., if $\bfd(j)\ne0$ for all $j\ne i$, then up to isomorphism, $\overbrace{S_Q(i)\oplus\dots\oplus S_Q(i)}^{\bfd(i)}$
is the unique representation of graded dimension $\bfd$.

By definition, a representation $M$ of $Q$ is nilpotent if it satisfies the following equivalent conditions.
\begin{enumerate}
\item The augmentation ideal of $KQ$ generated by the arrows acts nilpotently on $M$.
\item Every path of sufficiently large length acts trivially on $M$.
\item The closure of the $G_V$-orbit of $M$ contains $0$.
\item Every simple subquotient of $M$ is of the form $S_Q(i)$ for some $i\in I$.
\end{enumerate}
These conditions are automatic if $Q$ does not admit oriented cycles, i.e., if $KQ$ is finite-dimensional.
(Henceforth, by a cycle we always mean an oriented cycle.)
Note that an extension of a nilpotent representation by a nilpotent representation is nilpotent.
The subset $R_Q^{\nilp}(V)\subset R_Q(V)$ of nilpotent representations is connected, Zariski closed and $G_V$-stable.

For any arrow $h$ let $h^\op$ be the opposite arrow (i.e., $(h^\op)'=h''$ and $(h^\op)''=h'$).
Let $Q^\op=(I,\Omega^\op)$ be the opposite quiver where $\Omega^\op=\{h^\op\mid h\in\Omega\}$
and let $\overline{Q}$ be the double quiver $(I,H)$ where $H=\Omega\cup\Omega^\op$ (disjoint union).
The map $V\mapsto V^*$ is a duality between the category of representations of $Q$ and those of $Q^{\op}$.
We can also identify canonically $R_{Q^\op}(V)$ with the dual vector space of $R_Q(V)$ (by taking the trace of the product)
and $R_{\overline{Q}}(V)=R_Q(V)\times R_{Q^{\op}}(V)$ with the cotangent bundle $T^*(R_Q(V))$ of $R_Q(V)$.

By definition, the preprojective algebra $\Pi=\Pi_Q$ is the quotient of the path algebra of $\overline{Q}$
by the ideal generated by $\sum_{h\in\Omega}[h,h^{\op}]$.
Let $R_\Pi(V)$ be the space of $\Pi_Q$-module structures on $V$
such that $e_iV=V_i$ for all $i$. This is a closed subvariety of $R_{\overline Q}(V)$.
It is the fiber over $0$ of the moment map
\begin{equation} \label{eq: mmap}
T^*(R_Q(V))\rightarrow\oplus_{i\in I}\Lin(V_i).
\end{equation}
The projection
\[
\pi_Q:R_\Pi(V)\rightarrow R_Q(V)
\]
is the pullback with respect to the homomorphism (in fact, embedding) $KQ\rightarrow\Pi_Q$.

Let $\Lambda(V)=\Lambda_Q(V)=R_\Pi(V)\cap R_{\overline{Q}}^{\nilp}(V)$
be the Zariski closed, $G_V$-stable subvariety of nilpotent $\Pi$-modules.
This is Lusztig's nilpotent variety studied in \cite[\S12]{MR1088333}.
It is of pure dimension $\dim R_Q(V)$ and in fact a Lagrangian subvariety of $T^*(R_Q(V))$.

We denote by $\Irrcomp(V)$ the (finite) set of irreducible components of $\Lambda(V)$.
Since the group $G_V$ is connected, each irreducible component of $\Lambda(V)$ is $G_V$-stable.
Thus, $\Irrcomp(V)$ depends only on the graded dimension of $V$. We will therefore also use the notation
$\Irrcomp(\bfd)$ for any $\bfd\in\N I$.

For instance, if $\bfd$ is concentrated at a single vertex, then $R_\Pi(V)=\Lambda(V)=\{0\}$.
In particular, $\Irrcomp(\bfd)$ is a singleton in this case and we call the element of $\Irrcomp(\bfd)$ \emph{unicolor}.
We denote by $S(i)$ the (simple) $\Pi$-module whose graded dimension is the indicator function of $i\in I$.
By abuse of notation, we also denote by $S(i)$ the corresponding irreducible component.

If $Q$ does not have cycles, then for any $V\in\Vgr$, $R_Q(V)\times\{0\}$
and $\{0\}\times R_{Q^{\op}}(V)$ are irreducible components of $\Lambda(V)$.

Set
\[
\Irrcomp=\bigcup_{\bfd\in\N I}\Irrcomp(\bfd)\ \ \ \ \text{(disjoint union)}.
\]

Let $(\cdot,\cdot)$ be the bi-additive symmetric form on $\N I$ given by
\begin{equation} \label{eq: def[]}
(\bfd,\bfe)=\sprod{\bfd}{\bfe}_Q+\sprod{\bfd}{\bfe}_{Q^\op}=\sprod{\bfd}{\bfe}_Q+\sprod{\bfe}{\bfd}_Q.
\end{equation}
Note that if $\grdim V=\bfd$, then
\begin{equation} \label{eq: (dd)}
(\bfd,\bfd)=2(\dim G_V-\dim\Lambda(V)).
\end{equation}
By a result of Crawley-Boevey \cite[Lemma 1]{MR1781930} for any $\Pi$-modules $M$ and $N$
of graded dimensions $\bfd$ and $\bfe$ respectively we have
\begin{equation} \label{eq: CB}
\dim\Hom_\Pi(M,N)-\dim\Ext^1_\Pi(M,N)+\dim\Hom_\Pi(N,M)=(\bfd,\bfe).
\end{equation}
In particular,
\begin{equation} \label{eq: CB2}
\dim\Ext^1_\Pi(M,M)=2\dim\End_\Pi(M)-(\bfd,\bfd).
\end{equation}
In fact, (see \cite[\S4.2]{MR3270589} or \cite[\S8]{MR2360317} and the references therein)
we have a functorial isomorphism
\begin{equation} \label{eq: extdalty}
\Ext_\Pi^1(N,M)\simeq\Ext_\Pi^1(M,N)^*
\end{equation}
and an injection\footnote{If no connected component of $Q$ is of Dynkin type, then
in fact $\Ext_\Pi^2(N,M)\simeq\Hom_\Pi(M,N)^*$.}
\begin{equation} \label{eq: extsrjct}
\Ext_\Pi^2(N,M)\hookrightarrow\Hom_\Pi(M,N)^*.
\end{equation}

Note that we can rewrite \eqref{eq: CB} as
\begin{align*}
\dim\Ext_\Pi^1(M,N)=&(\dim\Lambda(V\oplus V')-\dim G_{V\oplus V}\cdot(M\oplus N))-\\&
(\dim\Lambda(V)-\dim G_V\cdot M)-(\dim\Lambda(V')-\dim G_V\cdot N)
\end{align*}
where $M\in R_\Pi(V)$ and $N\in R_\Pi(V')$.

For any $C_1,C_2\in\Irrcomp$ we set
\begin{gather*}
\hom_\Pi(C_1,C_2)=\min\{\dim\Hom_\Pi(x_1,x_2)\mid x_1\in C_1,x_2\in C_2\},\\
\ext_\Pi^1(C_1,C_2)=\min\{\dim\Ext^1_\Pi(x_1,x_2)\mid x_1\in C_1,x_2\in C_2\}.
\end{gather*}
We recall that $\dim\Hom_\Pi(x_1,x_2)$ and $\dim\Ext_\Pi^1(x_1,x_2)$ are upper semicontinuous functions,
so that the minimal loci
$\argmin_{C_1\times C_2}\dim\Hom_\Pi(x_1,x_2)$ and $\argmin_{C_1\times C_2}\dim\Ext^1_\Pi(x_1,x_2)$ are open in $C_1\times C_2$. By \eqref{eq: CB} we have
\[
\hom_\Pi(C_1,C_2)-\ext^1_\Pi(C_1,C_2)+\hom_\Pi(C_2,C_1)=(\bfd,\bfe)
\]
for any $C_1\in\Irrcomp(\bfd)$, $C_2\in\Irrcomp(\bfe)$.

The preprojective algebra is isomorphic to its opposite algebra.
(More precisely, $\Pi_Q^{\op}\simeq\Pi_{Q^{\op}}=\Pi_Q$, the first isomorphism takes a path to its opposite.)
This gives a self-duality $M\mapsto M^*$ on finite-dimensional $\Pi$-modules.
It yields an isomorphism $\Lambda(V)\rightarrow\Lambda(V^*)$,
and a bijection $\Irrcomp(V)\rightarrow\Irrcomp(V^*)$, which we also denote by $C\mapsto C^*$.
We therefore get an involution on $\Irrcomp(\bfd)$ for any $\bfd\in\N I$.

Up to isomorphism, the preprojective algebra, and hence its nilpotent varieties, depend only on $\overline{Q}$
(or in other words, on the underlying graph of $Q$, without orientation) \cite[\S12.15]{MR1088333}.
These isomorphisms depend on certain choices of signs and square-roots of $-1$,
but the induced bijection on the set of irreducible components does not depend on this choice.

\subsection*{The Dynkin case} (\cite{MR0332887} \cite{MR0393065} \cite{MR545362}  \cite[\S14]{MR1088333}) \label{sec: Dynkin}

In this subsection we assume that $Q$ is of Dynkin type, i.e., it satisfies the following equivalent conditions.
\begin{enumerate}
\item The underlying graph of $Q$ is a simply laced Dynkin diagram.
\item The quadratic form $\sprod{\bfd}{\bfd}_Q$ is positive-definite.
\item $R_Q(V)$ admits an open $G_V$-orbit for all $V\in\Vgr$.
\item $R_Q(V)$ admits finitely many $G_V$-orbits for all $V\in\Vgr$.
\item Up to isomorphism, there are only finitely many indecomposable representations of $Q$, i.e., $KQ$ is representation finite.
\item $\Pi_Q$ is finite-dimensional.
\end{enumerate}
The map $M\mapsto\grdim M$ is a bijection between the equivalence classes
of indecomposable representations of $Q$ and the set
\[
\Psi=\{\bfd\in\N I\mid\sprod{\bfd}{\bfd}_Q=1\}.
\]
We can identify $\Psi$ with the set of positive roots in the root system corresponding to the Dynkin diagram.
The set of simple roots is $I$ itself.

For $\beta\in\Psi$ denote by $M_Q(\beta)$ the indecomposable representation with graded dimension $\beta$.
It is simple if and only if $\beta\in I$.
Every indecomposable representation $M$ is a brick (i.e., $\End_Q(M)=K$) and is rigid (i.e., $\Ext_Q^1(M,M)=0$).

We write a typical element of $\Mult:=\N \Psi$ as a formal sum $\m=\beta_1+\dots+\beta_k$, $\beta_i\in\Psi$.
Then,
\[
\m=\beta_1+\dots+\beta_k\mapsto M_Q(\m)=M_Q(\beta_1)\oplus\dots\oplus M_Q(\beta_k)
\]
is a bijection between $\Mult$ and the isomorphism classes of representations of $Q$.
Extend the inclusion $\Psi\rightarrow\N I$ to an additive map
\[
\grdim:\Mult\rightarrow\N I.
\]
Then, $\grdim M_Q(\m)=\grdim\m$.
Moreover, $\Lambda(V)=R_\Pi(V)$ (i.e., every $\Pi$-module is nilpotent)
and $\Lambda(V)$ is the disjoint union of the conormal bundles to the (finitely many) $G_V$-orbits
in $R_Q(V)$.
Thus, the irreducible components of $\Lambda(V)$ are the closure of these conormal bundles.
Therefore, $\Irrcomp(\bfd)$ is in bijection with the isomorphism classes of $Q$-representations of graded dimension $\bfd$.
This gives rise to a bijection
\[
\prm_Q:\Mult\rightarrow\Irrcomp
\]
such that $\Irrcomp(\bfd)=\prm_Q(\Mult(\bfd))$
where $\Mult(\bfd)=\{\m\in\Mult\mid\grdim\m=\bfd\}$.
We denote by $\cn_Q(\m)$ the conormal bundle of the orbit of $M_Q(\m)$.
Thus, if $\grdim V=\grdim\m$, then $\cn_Q(\m)$ is the inverse image of the $G_V$-orbit of $M_Q(\m)$ under $\pi_Q$ and $\prm_Q(\m)$ is the Zariski closure of $\cn_Q(\m)$ in $\Lambda(V)$. Of course, we also have
a bijection
\[
\prm_{Q^\op}:\Mult\rightarrow\Irrcomp
\]
since the preprojective algebra pertaining to $Q^{\op}$ coincides with $\Pi$.

In particular, $\m\in\N I$ (i.e., all the roots in $\m$ are simple) if and only if
$\prm_Q(\m)=\pi_Q^{-1}(\{0\})=\{0\}\times R_{Q^{\op}}(V)$ if and only if
$\prm_{Q^{\op}}(\m)=R_Q(V)\times\{0\}$ where $V\in\Vgr(\grdim\m)$.

Also, for any $\beta\in\Psi$ we have $\prm_Q(\beta)=R_Q(V)\times\{0\}$
(for $V\in\Vgr(\beta)$), i.e., the $G_V$-orbit of $M_Q(\beta)$ is open in $R_Q(V)$.

We remark that in the context of a Lie group $G$ acting linearly,
with finitely many orbits on a vector space $V$,
the fiber of $0$ of the moment map $T^*(V)\rightarrow\operatorname{Lie}(G)^*$ was considered in \cite{MR0390138},
where it was shown to be of pure dimension $\dim V$.

We have
\begin{equation} \label{eq: dMQm}
M_Q(\m)^*\simeq M_{Q^{\op}}(\m)
\end{equation}
and consequently,
\[
\prm_Q(\m)^*=\prm_{Q^{\op}}(\m)
\]
for any $\m\in\Mult$.

\section{Extensions} \label{sec: binary}
For the rest of the paper we fix a quiver $Q$.

Let $V^i\in\Vgr(\bfd_i)$, $i=1,2$ and let $V\in\Vgr(\bfd)$ with $\bfd=\bfd_1+\bfd_2$.
For any $G_{V^1}\times G_{V^2}$-stable constructible subset $S$ of $\Lambda(V^1)\times\Lambda(V^2)$
let
\begin{align*}
\Exts_2(S)=\{x\in R_\Pi(V)\mid \,&\exists\text{ a short exact sequence }\\
&0\rightarrow x_2\rightarrow x\rightarrow x_1\rightarrow 0\text{ with }(x_1,x_2)\in S\}.
\end{align*}
It is a $G_V$-stable, constructible subset of $\Lambda(V)$.

Let $C_i\in\Irrcomp(V^i)$, $i=1,2$ and take
\begin{equation} \label{def: S}
S=\argmin_{C_1\times C_2}\dim\Ext_{\Pi}^1(x_1,x_2).
\end{equation}
Then, $S$ is a nonempty open subset of $C_1\times C_2$ (hence irreducible) and $\Exts_2(S)$
is irreducible (\cite[Theorem 1.3(ii)]{MR1944812}).
We denote its Zariski closure by $C_1*C_2$.

Clearly,
\begin{equation} \label{eq: ***}
(C_1*C_2)^*=C_2^**C_1^*.
\end{equation}

\begin{theorem} \label{thm: app}
$C_1*C_2$ is an irreducible component of $\Lambda(V)$.
Moreover,
\begin{equation} \label{eq: densE}
\begin{gathered}
\text{if $\emptyset\ne S'\subset S$ is open and $G_{V^1}\times G_{V^2}$-stable,}\\
\text{then $\Exts_2(S')$ is dense in $C_1*C_2$.}
\end{gathered}
\end{equation}
\end{theorem}

\begin{proof}
It will be more convenient to ``forget'' about the grading of $V$.
Thus, instead of $R_\Pi(V)$ we consider the larger variety $\tilde R_\Pi(V)$ of $\Pi$-module structure on $V$
(i.e., algebra homomorphisms $\Pi\rightarrow\Lin(V)$)
such that for all $i\in I$, $e_i$ acts as a projection onto a $\bfd(i)$-dimensional, but otherwise arbitrary, subspace of $V$.
Fixing a generating set of $\Pi$ of size $g$, we can identify $\tilde R_\Pi(V)$ with a closed subvariety of $\Lin(V)^g$
which is stable under $\GL(V)$ (acting diagonally by conjugation).
Let $\tilde\Lambda(V)\subset\tilde R_\Pi(V)$ be the closed, connected, $\GL(V)$-stable subvariety of nilpotent modules.
Clearly,\footnote{If $H$ is a subgroup of $G$ and $X$ is an $H$-set, then we write $G\overset{H}{\times}X$ for the $G$-set
$(G\times X)/H$ where $H$ acts freely on $G\times X$ by $(g,x)h=(gh,h^{-1}x)$.}
\[
\tilde R_\Pi(V)=\GL(V)\overset{G_V}{\times}R_\Pi(V)\text{ and }\tilde\Lambda(V)=\GL(V)\overset{G_V}{\times}\Lambda(V).
\]
Thus, $\tilde\Lambda(V)$ is of pure dimension $\delta(\bfd)=\dim\Lambda(V)+\dim\GL(V)-\dim G_V$
and the irreducible components of $\tilde\Lambda(V)$ and $\Lambda(V)$ are in natural bijection.
Note that by \eqref{eq: (dd)},
\begin{equation} \label{eq: (ddd)}
(\bfd,\bfd)=2(\dim\GL(V)-\delta(\bfd)).
\end{equation}
Let $\tilde C_i=\GL(V^i)\cdot C_i$ be the irreducible component of $\tilde\Lambda(V^i)$ corresponding to $C_i$, $i=1,2$.
Write
\[
\tilde S=(\GL(V^1)\times\GL(V^2))\cdot S\subset\tilde C_1\times\tilde C_2,
\ \ \widetilde{S'}=(\GL(V^1)\times\GL(V^2))\cdot S'\subset\tilde S
\]
and $\tilde\Exts=\tilde\Exts_2(\widetilde{S'})=\GL(V)\cdot\Exts_2(S')$.
It is enough to show that the dimension of $\tilde\Exts$ is $\delta(\bfd)$.

For simplicity, take $V=V^1\oplus V^2$.
Let $Z$ be the subset of $\tilde R_\Pi(V)$ (and in fact, of $\tilde\Lambda(V)$)
consisting of the module structures on $V$ for which
$V^2$ is a submodule and the pair of induced structures on $V^1=V/V^2$ and $V^2$ is an element of $\widetilde{S'}$.
(With respect to the decomposition $V=V^1\oplus V^2$, any element of $Z$, viewed in $\Lin(V)^g$,
is block lower triangular in each coordinate.)
Thus, $\tilde\Exts=\GL(V)\cdot Z$.

As explained in \cite[\S4 and \S5]{MR1944812}, the dimensions of
$\Ext^1_\Pi(x_1,x_2)$ and $\Hom_\Pi(x_1,x_2)$ are constant (say $N_1$ and $N_0$) on $\tilde S$
and via the canonical map $p:Z\rightarrow\widetilde{S'}$, $Z$ becomes a vector bundle over $\widetilde{S'}$ of rank $N_1-N_0+\dim V^1\cdot\dim V^2$.
In particular, $Z$ is irreducible of dimension
\[
\dim\widetilde{S'}+N_1-N_0+\dim V^1\cdot\dim V^2=\delta(\bfd_1)+\delta(\bfd_2)+N_1-N_0+\dim V^1\cdot\dim V^2.
\]
In the case at hand, we have (by \eqref{eq: CB} and \eqref{eq: (ddd)})
\begin{equation} \label{eq: dimext1}
\begin{aligned}
\dim\Ext^1_\Pi(x_1,x_2)=\delta(\bfd)-\delta(\bfd_1)-\delta(\bfd_2)
-2\dim V^1\cdot\dim V^2+\\\dim\Hom_\Pi(x_1,x_2)+\dim\Hom_\Pi(x_2,x_1).
\end{aligned}
\end{equation}
It follows that $\tilde S\subset\argmin_{C_1\times C_2}\Hom_\Pi(x_2,x_1)$ and
\begin{equation} \label{eq: codim2}
\delta(\bfd)-\dim Z=\dim V^1\cdot\dim V^2-\hom_\Pi(C_2,C_1).
\end{equation}

Let $R$ be the unipotent radical of the parabolic subgroup of $\GL(V)$ preserving $V^1$.
Let $Y$ be the Zariski closure of $R\cdot Z$. Clearly, $Y$ is irreducible.
We show that $\dim Y=\delta(\bfd)$, which will clearly imply the theorem.
Let $X=R\times Z$ and consider the action map
\[
f:X\rightarrow Y.
\]
Identify $R$ with $\Lin(V^2,V^1)$ via $T\mapsto\sm 1T{}1$.
We claim that for any $(x_1,x_2)\in\widetilde{S'}$, the fiber $X_x$ of $f$ at $x=x_1\oplus x_2\in Z$ is $\Hom_\Pi(x_2,x_1)\times\{x\}$.
Indeed, the lower left corner is preserved under the action of $R$.
Therefore, if $T\cdot z=x$ with $T\in R$ and $z\in Z$, then $z$ is diagonal, which clearly implies that $z=x$
and $T\in\Hom_\Pi(x_2,x_1)$.

In particular, by \eqref{eq: codim2}
\[
\dim X_x=\dim R+\dim Z-\delta(\bfd)=\dim X-\delta(\bfd).
\]
It follows from \cite[Theorem 1.25]{MR3100243} that
\[
\dim Y\ge\dim X-\dim X_x=\delta(\bfd).
\]
Since evidently, $\dim Y\le\delta(\bfd)$, we conclude that $\dim Y=\delta(\bfd)$, as required.
\end{proof}

The theorem gives rise to a binary operation (also denoted by $*$) on $\Irrcomp$ such that
$\Irrcomp(\bfd_1)*\Irrcomp(\bfd_2)\subset\Irrcomp(\bfd_1+\bfd_2)$.

\begin{remark} \label{rem: BKT}
In the case where $C_1$ (or $C_2$) is unicolor,
Theorem \ref{thm: app}, together with Proposition \ref{prop: cancel} below, was proved by Lusztig (see \cite[\S12]{MR1088333}).
It was conjectured by him \cite{MR1182165} that this gives rise to the crystal graph defined by Kashiwara \cite{MR1115118}.
This was subsequently proved by Kashiwara and Saito \cite{MR1458969}.

A broader situation, from which the previous case can be deduced, was considered by Baumann--Kamnitzer--Tingley in \cite{MR3270589}.
It is pertaining to cases where $\hom_\Pi(C_2,C_1)=0$.
More precisely, let $(\tors,\torf)$ be a torsion pair (see [ibid., \S3.1]) in the category of nilpotent $\Pi$-modules.
Assume that $\tors(\bfd)=\Lambda(\bfd)\cap\tors$ and $\torf(\bfd)=\Lambda(\bfd)\cap\torf$ are open in $\Lambda(\bfd)$
for all $\bfd\in\N I$.
Let $\Tors(\bfd)$ (resp., $\Torf(\bfd)$) be the irreducible components of $\tors(\bfd)$ (resp., $\torf(\bfd)$), viewed
as subsets of $\Irrcomp(\bfd)$ (by taking the Zariski closure). Let
\[
\Tors=\bigcup_{\bfd}\Tors(\bfd),\ \Torf=\bigcup_{\bfd}\Torf(\bfd).
\]
Then, $(C_1,C_2)\mapsto C_1*C_2$ defines a bijection $\Torf\times\Tors\rightarrow\Irrcomp$ ([ibid., Theorem 4.4]).\footnote{Note that
the convention for the product in [ibid.] is opposite to ours.}
Many important examples of such torsion pairs are given in [ibid., \S5].
\end{remark}

We say that $C_1$ and $C_2$ strongly commute if there exist $x_i\in C_i$, $i=1,2$
such that $\Ext^1_{\Pi}(x_1,x_2)=0$.
(This is an open condition in $(x_1,x_2)\in\Lambda(V^1)\times\Lambda(V^2)$.)
We caution that in general, an irreducible component does not necessarily strongly commute with itself.

Let $V=V^1\oplus V^2$ and denote by $C_1\oplus C_2$ the Zariski closure of the set
\[
G_V\cdot\{x_1\oplus x_2\mid x_1\in C_1\text{ and }x_2\in C_2\}.
\]
This is an irreducible subset of $\Lambda(V)$. Clearly,
\begin{equation} \label{eq: contsimp}
C_1*C_2\supseteq C_1\oplus C_2.
\end{equation}

\begin{corollary}[cf.\ \cite{MR1944812}] \label{cor: sc}
$C_1$ and $C_2$ strongly commute if and only if $C_1\oplus C_2$ is an irreducible component of $\Lambda(V)$,
in which case $C_1\oplus C_2=C_1*C_2$.
\end{corollary}

Indeed, if $C_1$ and $C_2$ strongly commute then $\Exts_2(S)\subset C_1\oplus C_2$ and therefore
$C_1\oplus C_2=C_1*C_2$ is an irreducible component.
Conversely, if $C_1\oplus C_2$ is an irreducible component, then the argument of \cite[p. 216]{MR1944812}
shows that $C_1$ and $C_2$ strongly commute.

\begin{remark}
Corollary \ref{cor: sc} is a special case of a general result of Crawley-Boevey and Schr\"oer.
Namely, let $A$ be a finitely generated algebra and for any finite-dimensional vector space $V$
let $R_A(V)$ be the variety of $A$-module structures on $V$.
Let $V=V^1\oplus V^2$ and let $C_i$ be irreducible components of $R_A(V^i)$, $i=1,2$.
Define $C_1\oplus C_2$ as above.
Then, by \cite{MR1944812}, $C_1\oplus C_2$ is an irreducible component of $R_A(V)$ if and only
if there exist $x_i\in C_i$, $i=1,2$ such that $\Ext_A^1(x_1,x_2)=\Ext_A^1(x_2,x_1)=0$ .
(When $A$ admits orthogonal idempotents $e_i$, $i\in I$ such that $\sum e_i=1$, we immediately deduce a version with graded dimensions.)
To obtain Corollary \ref{cor: sc} we take $A$ to be the quotient of $\Pi$ by a suitable power of the augmentation
ideal. (In the Dynkin case, $A$ is $\Pi$ itself.)
Of course, in the preprojective case, the condition $\Ext^1_{\Pi}(x_1,x_2)=0$ is symmetric by \eqref{eq: extdalty}.
\end{remark}

\section{Rigid modules} \label{sec: rigid}

Recall that by definition, a $\Pi$-module $x$ is rigid if $\Ext_\Pi^1(x,x)=0$.

Clearly, $x$ is rigid if and only if $x^*$ is rigid.

If $x=x_1\oplus\dots\oplus x_k$, then
\begin{equation} \label{eq: x12r}
x\text{ is rigid}\iff \Ext_\Pi^1(x_i,x_j)=0\text{ for all }i,j.
\end{equation}

For instance, if the graded dimension of $x$ is concentrated at $i\in I$, then $x$ is rigid.
Also, every projective module is rigid.

Let $V$ be a finite-dimensional $I$-graded vector space.
Let ${\bf R}_{\Pi}(V)$ be the \emph{scheme} of $\Pi$-modules on $V$
such that $e_iV=V_i$ for all $i\in I$.
By Voigt's lemma (which is valid for any finitely generated algebra --
see \cite[\S1.1]{MR0376769} for a standard reference)
for any $x\in R_\Pi(V)$ with $G_V$-orbit $\orb(x)$ we have
\begin{equation} \label{eq: Voigt}
\Ext_\Pi^1(x,x)\simeq N_x^{{\bf R}_{\Pi}(V)}(\orb(x))=T_x{\bf R}_\Pi(V)/T_x\orb(x)
\end{equation}
where $T_x{\bf R}_\Pi(V)$ is the Zariski tangent space of the scheme ${\bf R}_\Pi(V)$ at $x$.
In particular, $x$ is rigid if and only if $\orb(x)$ is an open subscheme of ${\bf R}_{\Pi}(V)$.
(Of course, we can forgo the condition $e_iV=V_i$ and consider the $\GL(V)$-orbit instead.)

On the other hand, by \eqref{eq: CB2}, $x$ is rigid if and only if
\[
2\dim\End_\Pi(x)=(\bfd,\bfd).
\]
Since
\[
\dim\orb(x)=\dim G_V-\dim\End_\Pi(x),
\]
we also have
\[
\dim\Ext_\Pi^1(x,x)=2(\dim\Lambda(V)-\dim\orb(x))
\]
(by \eqref{eq: (dd)}). Comparing with \eqref{eq: Voigt} we get
\[
\tfrac12\dim\Ext_\Pi^1(x,x)=\dim\Lambda(V)-\dim\orb(x)=\dim T_x{\bf R}_\Pi(V)-\dim\Lambda(V).
\]
In particular, if $x\in\Lambda(V)$ then
\[
\dim\Ext_\Pi^1(x,x)=2\codim\orb(x)\ \ \ \text{ (codimension in $\Lambda(V)$).}
\]
Therefore, the following conditions are equivalent for $x\in\Lambda(V)$.
\begin{enumerate}
\item $x$ is rigid.
\item $\orb(x)$ is open in $\Lambda(V)$.
\item The Zariski closure $\overline{\orb(x)}$ is an irreducible component of $\Lambda(V)$.
\item $\dim\End_\Pi(x)=\dim G_V-\dim\Lambda(V)$.
\item $\dim\End_\Pi(x)\le\dim G_V-\dim\Lambda(V)$.
\item $x$ is a smooth point in the scheme ${\bf R}_\Pi(V)$ (in particular, it lies in a unique irreducible
component $C$ of $R_\Pi(V)$) and $\dim C=\dim\Lambda(V)$.\footnote{The condition $\dim C=\dim\Lambda(V)$
is redundant in the case $Q$ is of Dynkin type.}
\end{enumerate}

\begin{definition} \label{def: irrigid}
An irreducible component $C$ of $\Lambda(V)$ is called rigid if it contains a rigid module, or equivalently
it contains a (necessarily unique) open $G_V$-orbit.
\end{definition}
In this case, the open orbit in $C$ consists of the rigid modules in $C$.

For instance, any unicolor $C\in\Irrcomp$ is rigid.

If $C$ is rigid, then $C^*$ is rigid.

\begin{remark}
Suppose that $Q$ is of Dynkin type.
Then, an irreducible component $C$ of $\Lambda(V)$ is rigid if and only if
the scheme ${\bf R}_\Pi(V)$ is generically smooth (or equivalently, generically reduced) at $C$.\footnote{The latter
equivalence holds for any irreducible component of a scheme of finite type over $K$.}

If $C=\prm_Q(\m)$, then $C$ is rigid if and only if $\cn_Q(\m)$ contains an open $G_V$-orbit.
In this case, $\prm_{Q^{\op}}(\m)=\prm_Q(\m)^*$ is also rigid.

For any $V\in\Vgr$, $R_Q(V)\times\{0\}$ and $\{0\}\times R_{Q^{\op}}(V)$ are rigid.
In particular, $\prm_Q(\beta)$ is rigid for any $\beta\in\Psi$.
\end{remark}

In general, we have a bijection between the set of rigid irreducible components of $\Lambda(V)$
and the set of $G_V$-orbits of rigid modules in $\Lambda(V)$.

The first examples of non-rigid irreducible components were given by Leclerc in \cite{MR1959765}.
(We will recall it in the second part of the paper.)

The role of rigid modules and irreducible components was highlighted in the work
of Geiss, Leclerc and Schr\"oer (e.g., \cite{MR2242628} and also below).

\begin{remark} \label{rem: CCrigid}
Let $C_1,C_2,C\in\Irrcomp$.
\begin{enumerate}
\item By \eqref{eq: x12r}, if $C_1$ and $C_2$ strongly commute, then $C_1*C_2=C_1\oplus C_2$ is rigid
if and only if $C_1$ and $C_2$ are rigid.
\item If $C$ is rigid, then $C$ strongly commutes with itself. The converse, however is not true as Leclerc's example shows.
\item It is possible for $C*C$ to be rigid even if $C$ itself is not. We will give an example in the second part of the paper.
\end{enumerate}
\end{remark}

\section{Cancellation} \label{sec: cancel}
We show that rigid irreducible components are cancellative.
(We do not know whether non-rigid irreducible components are cancellative as well.)

Fix a rigid $C_1\in\Irrcomp$ and a rigid element $x_1$ in $C_1$.
For any $V\in\Vgr(\bfd)$ and an irreducible component $C\in\Irrcomp(V)$, denote by $C'$ the (possibly empty) constructible, $G_V$-stable
subset of $C$ consisting of the elements $x$ that admit $x_1$ as a quotient of $\Pi$-modules.
Denote by $\Irrcomp_{C_1\cvr}(V)$ or $\Irrcomp_{C_1\cvr}(\bfd)$ the set of irreducible components $C\in\Irrcomp(V)$
for which $C'$ is dense in $C$.
Finally, set
\[
\Irrcomp_{C_1\cvr}=\cup_{\bfd\in\N I}\Irrcomp_{C_1\cvr}(\bfd)\subset\Irrcomp.
\]
The following proposition generalizes a result of Lusztig in the case where $C_1$ is unicolor.

\begin{proposition} \label{prop: cancel}
The map $C_2\in\Irrcomp\mapsto C_1*C_2\in\Irrcomp$ defines a bijection $\Irrcomp\rightarrow\Irrcomp_{C_1\cvr}$.
Its inverse takes $C\in\Irrcomp_{C_1\cvr}(V)$ to the closure of the set of $\Pi$-modules
$\{\Ker\phi\mid x\in S',\phi:x\twoheadrightarrow x_1\}$,
for any open, nonempty $G_V$-stable subset $S'$ of
\[
S=\{x\in C'\mid\dim\Hom_\Pi(x,x_1)=\hom_\Pi(C,C_1)\}.
\]

Similarly, $C_2\mapsto C_2*C_1$ is one-to-one and its image consists of the set of irreducible components $C$
for which the subset $C''$ consisting of the $x$'s that admit $x_1$ as a submodule is dense in $C$.
The inverse map is given by taking the closure of the set of $\Pi$-modules $\{\Coker\phi\mid x\in S',\phi:x_1\hookrightarrow x\}$
for any open, nonempty $G_V$-stable subset $S'$ of $\{x\in C''\mid\dim\Hom_\Pi(x_1,x)=\hom_\Pi(C_1,C)\}$.
\end{proposition}

\begin{proof}
We only need to prove the first statement, since the second one can then be deduced by passing to $C^*$.

Clearly, $C_1*C_2\in\Irrcomp_{C_1\cvr}$ for any $C_2\in\Irrcomp$.
Conversely, suppose that $C\in\Irrcomp_{C_1\cvr}(V)$ with $\grdim V=\bfd$.
Write $\bfd=\bfd_1+\bfd_2$ where $C_1\in\Irrcomp(\bfd_1)$.
Let $d$ be the total dimension of $\bfd$. Similarly for $d_1$ and $d_2$.

As in the proof of Theorem \ref{thm: app} we prove an analogous (but equivalent) statement for
$\tilde\Lambda(V)=\GL(V)\overset{G_V}{\times}\Lambda(V)$.
Thus, we replace $C$, $S$ and $S'$ by their image under the action of $\GL(V)$
(so that $C$ is an irreducible component of $\tilde\Lambda(V)$) and similarly for $C_1$.
Fix a $d_2$-dimensional subspace $V_2\subset V$ and let $V_1=V/V_2$.
We view $C_1$ as an irreducible component of $\tilde\Lambda(V_1)$.

Let
\[
X=\{(x,\varphi)\mid x\in S',\varphi\in\Hom_\Pi(x,x_1),\Ker\varphi=V_2\},
\]
and consider the morphism
\[
f:X\rightarrow\tilde\Lambda(V_1)\times\tilde\Lambda(V_2)
\]
that takes $(x,\varphi)\in X$ to the quotient and submodule structures induced by $x$ on $V_1$ and $V_2$.

Let $Y$ be the closure of the image of $f$.
We claim that $Y=C_1\times C_2$ where $C_2$ is the unique irreducible component of $\tilde\Lambda(V_2)$ such that $C=C_1*C_2$.

We first show that $X$ is irreducible of dimension $\delta(\bfd)+\hom_\Pi(C,C_1)-d_1d_2$
where $\delta(\bfd)=\dim\tilde\Lambda(V)$.
Let
\[
C^\circ=\argmin_{x\in C}\dim\Hom_\Pi(x,x_1),
\]
which is a nonempty open subset of $C$. Consider
\[
E=\{(x,\varphi)\mid x\in C^\circ,\varphi\in\Hom_\Pi(x,x_1)\}.
\]
Thus, $E$ is the fiberwise kernel of the morphism
\[
\xi:C^\circ\times\Lin(V,V_1)\rightarrow C^\circ\times\Lin(V,V_1)^N
\]
of trivial vector bundles over $C^\circ$ given by
\[
\xi(x,\varphi)=(x,\big(\varphi\circ x(a_i)-x_1(a_i)\circ\varphi\big)_{i=1}^N)
\]
where $a_1,\dots,a_N$ are fixed generators of $\Pi$.
Since $\xi$ is of constant rank (by the definition of $C^\circ$), $E$ is a vector bundle over $C^\circ$ by a standard result
(\cite[Proposition 1.7.2]{MR1428426}).\footnote{We thank Yakov Varshavsky for providing us this reference.}

We may consider $V$ and $V_1$ as constant sheaves of vector spaces $\mathcal{F}$ and $\mathcal{G}$
over $E$. The map
\[
\Phi:\mathcal{F}\rightarrow\mathcal{G}
\]
given by $\varphi$ is a morphism of sheaves over $E$. Let $\mathcal{B}=\Coker\Phi$.
Note that the fiber $\mathcal{B}_{(x,\varphi)}$ is $\Coker\varphi$.
Since $C\in\Irrcomp_{C_1\cvr}$, the complement $E^\circ$ of the support of $\mathcal{B}$ in $E$ is nonempty
(and open). Note that
\[
E^\circ=\{(x,\varphi)\mid x\in C^\circ,\varphi\in\Hom_\Pi(x,x_1),\varphi\text{ is surjective}\}.
\]
Thus, $S$, which is the image of $E^\circ$ under the canonical map $E\rightarrow C^\circ$, is open.
Let
\[
E'=\{(x,\varphi)\in E^\circ\mid x\in S'\},
\]
which is open in $E$.
Let $Z$ be the Grassmannian variety of $d_2$-dimensional subspaces of $V$.
Let $\alpha:E'\rightarrow Z$ and $\beta:\GL(V)\rightarrow Z$ be the morphisms
\[
\alpha(x,\varphi)=\Ker\varphi,\ \ \beta(g)=g^{-1}(V_2)
\]
and let $F$ be the pull back $E'\times_Z\GL(V)$.
Then, $F$ is a principal $H$-bundle over $E'$, where $H$ is the parabolic subgroup of $\GL(V)$ stabilizing $V_2$.
In particular, $F$ is irreducible of dimension
\begin{align*}
\dim E'+d^2-\dim Z&=\dim C+\hom_\Pi(C,C_1)+d^2-d_1d_2\\&=\delta(\bfd)+\hom_\Pi(C,C_1)+d^2-d_1d_2.
\end{align*}
Note that $X$ is equal to the fiber of the map $F\rightarrow\GL(V)$ over the identity.
Now, the maps $\alpha$ and $\beta$ are $\GL(V)$-equivariant with respect to the $\GL(V)$-action on $E'$
given by $g(x,\varphi)=(g\cdot x,\varphi\circ g^{-1})$, the right regular action of $\GL(V)$ on itself,
and the usual action of $\GL(V)$ on $Z$. Thus, $\GL(V)$ acts freely on $F$ and the action map gives rise to an isomorphism
\[
\GL(V)\times X\simeq F.
\]
It follows that $X$ is irreducible of dimension $\delta(\bfd)+\hom_\Pi(C,C_1)-d_1d_2$, as claimed.

Next, we analyze the fibers of $f$ as in \cite[\S4 and \S5]{MR1944812}, except that we also
have to take $\varphi$ into account.
Suppose that $y=(y_1,y_2)$ is in the image of $f$. Then, $y_1$ lies in the orbit of $x_1$
and the fiber $X_y$ is given by the product of a vector space of dimension
$\dim\Ext_\Pi^1(y_1,y_2)-\dim\Hom_\Pi(y_1,y_2)+d_1d_2$
with the $\Aut_\Pi(x_1)$-torsor of isomorphisms of $\Pi$-modules between $y_1$ and $x_1$.
In particular, $X_y$ is irreducible of dimension
\begin{gather*}
\dim\Ext_\Pi^1(x_1,y_2)-\dim\Hom_\Pi(x_1,y_2)+d_1d_2+\dim\Aut_\Pi(x_1)\\=
\delta(\bfd)-\delta(\bfd_1)-\delta(\bfd_2)-d_1d_2+\dim\Hom_\Pi(y_2,x_1)+\dim\Aut_\Pi(x_1)
\end{gather*}
(by \eqref{eq: dimext1}).
Note that since $x_1$ is rigid, for any $(x,\varphi)\in X_y$ we have a short exact sequence
\[
0\rightarrow\Hom_\Pi(y_1,x_1)\rightarrow\Hom_\Pi(x,x_1)\rightarrow\Hom_\Pi(y_2,x_1)\rightarrow0
\]
and hence
\[
\hom_\Pi(C,C_1)=\dim\Hom_\Pi(x,x_1)=\dim\End_\Pi(x_1)+\dim\Hom_\Pi(y_2,x_1).
\]
Thus,
\[
\dim X_y=\delta(\bfd)-\delta(\bfd_1)-\delta(\bfd_2)-d_1d_2+\hom_\Pi(C,C_1),
\]
which is independent of $y$.

It follows that
\[
\dim Y=\dim X-\dim X_y=\delta(\bfd_1)+\delta(\bfd_2),
\]
and hence $Y$ is an irreducible component of $\tilde\Lambda(V_1)\times\tilde\Lambda(V_2)$.
Since $Y$ is contained in $C_1\times\tilde\Lambda(V_2)$, it is necessarily of the form $Y=C_1\times C_2$ for some irreducible component
$C_2$ of $\tilde\Lambda(V_2)$, as claimed.

Clearly, for any $\GL(V_1)\times\GL(V_2)$-stable subset $U$ of $\tilde\Lambda(V_1)\times\tilde\Lambda(V_2)$
we have $\Exts_2(U)\supset p(f^{-1}(U))$ where $p:F\rightarrow C$ is the composition of
the canonical maps $F\rightarrow E'\hookrightarrow E\rightarrow C^\circ\hookrightarrow C$.
(Recall that $X$ is contained in $F$ and $F=\GL(V)\cdot X$,
and note that the restriction of $p$ to $X$ is the first projection $X\rightarrow S'$.)
Since these maps are open and $\GL(V)$-equivariant, if $U$ is open, then $\Exts_2(U)$ contains the open set
$p(\GL(V)\cdot f^{-1}(U))$ of $C$.
Taking $U=\argmin_{C_1\times C_2}\Ext_\Pi^1(x_1,x_2)$ we infer that $C_1*C_2\supset C$ and hence $C_1*C_2=C$.

Finally, suppose that $C=C_1*C_2'$ for some irreducible component $C_2'$ of $\tilde\Lambda(V_2)$.
Let $U=\argmin_{x_2\in C_2'}\dim\Ext_\Pi^1(x_1,x_2)$, which is a nonempty open subset of $C_2'$.
Let $W$ be the subset of $\tilde\Lambda(V)$ consisting
of the $\Pi$-module structures on $V$ for which $V_2$ is a submodule in $U$ and the induced structure
on $V_1$ is $x_1$. Then, $W$ is a vector bundle over $U$
and by Theorem \ref{thm: app}, $\GL(V)\cdot W$ is a constructible dense subset of $C$. Hence, $W^\circ=W\cap S'$
is nonempty and open in $W$ (since $S'$ is $\GL(V)$-stable).
Let $p:W\rightarrow U$ be the canonical map.
Then, $p(W^\circ)$ is open in $U$ (hence in $C_2'$) and $\{x_1\}\times p(W^\circ)$ is contained in the image of $f$.
Therefore, $C'_2=C_2$.

The proposition follows.
\end{proof}

\section{Commutativity} \label{sec: comm}
Next, we discuss the lack of commutativity of the $*$ operation.

This transpires already in the simplest example of $Q=A_2$, $C_i=S(i)$, $i=1,2$.
Writing $I=\{1,2\}$ with $1\rightarrow 2$ and $\Psi=\{\alpha_1,\alpha_2,\beta=\alpha_1+\alpha_2\}$, we have
\[
S(1)*S(2)=\prm_Q(\{\beta\})\text{ but }S(2)*S(1)=\prm_{Q^{\op}}(\{\beta\})=\prm_Q(\{\alpha_1\}+\{\alpha_2\}).
\]

We say that two irreducible components $C_1,C_2\in\Irrcomp$ \emph{weakly commute} if
\[
C_1*C_2=C_2*C_1.
\]
Recall that, as the name suggests, strong commutativity implies weak commutativity by Corollary \ref{cor: sc}.
(See Corollary \ref{cor: scc} below for a more precise statement in the case where $Q$ is of Dynkin type.)
By Remark \ref{rem: CCrigid}, the converse is not true in general, even if $C_1=C_2$.
However, the converse holds in the case where $C_1$ or $C_2$ is rigid.

\begin{proposition} \label{prop: ws}
Let $C_1$, $C_2$ be irreducible components in $\Irrcomp$. Suppose that $C_1$ or $C_2$ is rigid.
Then,
\[
\text{$C_1$ and $C_2$ weakly commute if and only if $C_1$ and $C_2$ strongly commute.}
\]
If $x_1\in C_1$ (say) is rigid, then the commutativity is equivalent to
\begin{equation} \label{eq: extrho}
\Ext_\Pi^1(x_1,x_2)=0\text{ for some }x_2\in C_2.
\end{equation}
(This condition is open in $x_2$.)
\end{proposition}

The equivalence of strong commutativity with \eqref{eq: extrho} (in the case where $C_1$ is rigid) is clear.
In order to show that weak commutativity implies strong one in the case at hand, we prove the following general result.

\begin{lemma} \label{lem: 2ses}
Let $A$ be any algebra. Suppose that we have two short exact sequences of finite-dimensional $A$-modules
\begin{equation} \label{eq: ses}
\begin{gathered}
0\rightarrow x_2\rightarrow x\rightarrow x_1\rightarrow0\\
0\rightarrow x_1\rightarrow x\rightarrow x_3\rightarrow0
\end{gathered}
\end{equation}
Assume that
\begin{enumerate}
\item $x_1$ is rigid.
\item $\dim\Hom_A(x_1,x_2)=\dim\Hom_A(x_1,x_3)$.
\item $\dim\Hom_A(x_2,x_1)=\dim\Hom_A(x_3,x_1)$.
\end{enumerate}
Then, the short exact sequences \eqref{eq: ses} split (and hence $x_3\simeq x_2$).
\end{lemma}

\begin{proof}
The short exact sequences \eqref{eq: ses} and the rigidity of $x_1$ yield two exact sequences
\begin{multline*}
0\rightarrow\Hom_A(x_1,x_2)\rightarrow\Hom_A(x_1,x)\xrightarrow{\alpha}\End_A(x_1)\xrightarrow{\beta}
\Ext_A^1(x_1,x_2)\xrightarrow{\gamma}\\\Ext_A^1(x_1,x)\rightarrow0
\end{multline*}
and
\[
0\rightarrow\End_A(x_1)\rightarrow\Hom_A(x_1,x)\rightarrow\Hom_A(x_1,x_3)\rightarrow0.
\]
Comparing dimensions and taking into account our hypothesis, we get
\[
\dim\Ext_A^1(x_1,x_2)=\dim\Ext^1_A(x_1,x).
\]
Hence, $\gamma$ is an isomorphism.
Therefore $\beta=0$, so that $\alpha$ is surjective. Thus, the first sequence in \eqref{eq: ses} splits.
By a dual argument, the second one also splits.
\end{proof}

Proposition \ref{prop: ws} follows from Lemma \ref{lem: 2ses}.
Indeed, we may assume without loss of generality that $C_1$ is rigid.
Let $x_1$ be a rigid element in $C_1$ and suppose that $C=C_1*C_2=C_2*C_1$. Let
\[
S_2'=\argmin_{x_2\in C_2}\dim\Ext_\Pi^1(x_1,x_2)=\argmin_{x_2\in C_2}\dim\Ext_\Pi^1(x_2,x_1),
\]
an open, nonempty subset of $C_2$, and let $S'=\orb(x_1)\times S_2'$ and $S''=S_2'\times\orb(x_1)$.
Then, $\Exts_2(S')$ and $\Exts_2(S'')$ are open in $C$ and by \cite[Lemma 4.4]{MR1944812} the conditions of Lemma \ref{lem: 2ses} are satisfied
for every $x\in\Exts_2(S')\cap\Exts_2(S'')$. Therefore, $C\subset C_1\oplus C_2$ and hence, $C_1\oplus C_2=C$, as required.

Henceforth, if $C_1$ or $C_2$ is rigid, then we simply say that $C_1$ and $C_2$ commute
if they weakly (or equivalently, strongly) commute.

\section{Associativity} \label{sec: assoc}
The operation $*$ is not associative, already for $Q=A_2$.
\begin{example} \label{ex: nonasoc}
Let $C_i=S(i)$, $i=1,2$. Then,
\[
(C_1*C_2)*C_1=(C_1*C_2)\oplus C_1,\ \ C_1*(C_2*C_1)=C_1\oplus (C_2*C_1).
\]
However, $C_1*C_2\ne C_2*C_1$. Hence,
\[
(C_1*C_2)*C_1\ne C_1*(C_2*C_1).
\]
\end{example}

Nevertheless, it turns out that there is a useful cohomological criterion which guarantees associativity, as will be explained below.

For the next result, which is in the spirit of \cite[Theorem 1.3(ii)]{MR1944812}, let $A$ be any finitely generated algebra.
We consider the module varieties $\bmod_A(d)$ of $d$-dimensional $A$-modules, $d\ge0$.
Fix $d_1,d_2,d_3\ge0$ and let $S$ be a $\GL_{d_1}\times\GL_{d_2}\times\GL_{d_3}$-stable, constructible subset of
$\bmod_A(d_1)\times\bmod_A(d_2)\times\bmod_A(d_3)$.
Let $d=d_1+d_2+d_3$ and define
\[
\begin{split}
\Exts_3(S)=\{M\in\bmod_A(d)\mid \exists\text{ submodules }N_1\subset N_2\subset M\text{ such that }\\
\text{the isomorphism classes of }(M/N_2,N_2/N_1,N_1)\text{ belongs to }S\}.
\end{split}
\]

\begin{proposition} \label{prop: irrext3}
Let $\cpp$ be the cup product
\begin{equation} \label{eq: cuprod}
\cpp:\Ext^1_A(M_1,M_2)\otimes\Ext^1_A(M_2,M_3)\rightarrow\Ext_A^2(M_1,M_3).
\end{equation}
Suppose that $S$ is irreducible and that
\begin{enumerate}
\item For each $1\le i<j\le 3$, the dimensions of $\Ext_A^1(M_i,M_j)$ are constant on $S$.
\item For $(M_1,M_2,M_3)\in S$, the varieties
\[
\zerocap_{(M_1,M_2,M_3)}=\{(\phi,\psi)\in\Ext^1_A(M_1,M_2)\times\Ext^1_A(M_2,M_3)\mid\phi\cpp\psi=0\}
\]
defined by quadratic equations are irreducible of constant dimension.
(For instance, this is satisfied if $\cpp$ is trivial.)
\end{enumerate}
Then, $\Exts_3(S)$ is irreducible.
\end{proposition}

\begin{proof}
We recall some standard facts.
For any $A$-modules $M_1$ and $M_2$, the space $M=\Lin(M_1,M_2)$ is an $A$-$A$ bimodule and
\[
\Ext_A^*(M_1,M_2)=H^*(A,M)
\]
where the right-hand side is Hochschild cohomology.
In particular,
\[
\Ext^1_A(M_1,M_2)=\Der(A,M)/\Der^0(A,M)
\]
where $\Der(A,M)$ is the space of derivatives, namely,
\[
\Der(A,M)=\{d\in\Lin(A,M)\mid d(ab)=ad(b)+d(a)b\ \forall a,b\in A\}
\]
and $\Der^0(A,M)$ is the subspace of inner derivatives, i.e., those of the form $a\mapsto am-ma$ for some $m\in M$.
Also,
\[
\Ext^2_A(M_1,M_2)=\Fac(A\otimes A,M)/\Fac^0(A\otimes A,M)
\]
where $\Fac(A\otimes A,M)$ denotes the space of factor sets, i.e., bilinear maps $f:A\times A\rightarrow M$ satisfying
\[
a_1f(a_2,a_3)+f(a_1,a_2a_3)=f(a_1a_2,a_3)+f(a_1,a_2)a_3\ \ \forall a_1,a_2,a_3\in A
\]
and $\Fac^0(A\otimes A,M)$ is the image of the coboundary map
\[
g\in\Lin(A,M)\mapsto \partial g(a,b)=g(ab)-ag(b)-g(a)b
\]
whose kernel is $\Der(A,M)$.

Now let $M_1,M_2,M_3$ be three $A$-modules. For brevity we write
$\LLin_{i,j}=\Lin(M_i,M_j)$, $1\le i<j\le3$.
The bilinear map
\[
\Der(A,\LLin_{1,2})\times\Der(A,\LLin_{2,3})\rightarrow\Fac(A\otimes A,\LLin_{1,3})
\]
given by
\[
(f_1,f_2)\mapsto (f_2\otimes f_1)(a,b)=f_2(a)f_1(b)
\]
induces the cap product \eqref{eq: cuprod}.

Now, consider the closed subset $Z$ of $\bmod_A(d)$ consisting
of homomorphisms $A\rightarrow\Mat_d$ that are block lower triangular with respect to the decomposition $d=d_1+d_2+d_3$.
The diagonal blocks induce a morphism
\begin{equation} \label{eq: defpi}
\pi: Z\rightarrow\bmod_A(d_1)\times\bmod_A(d_2)\times\bmod_A(d_3).
\end{equation}
The fiber at $(M_1,M_2,M_3)$ is given by the triples
\[
(f_1,f_2,g)\in\Der(A,\LLin_{1,2})\times\Der(A,\LLin_{2,3})\times\Lin(A,\LLin_{1,3})
\]
satisfying
\[
f_2\otimes f_1=\partial g.
\]
This is an affine bundle (with fiber $\Der(A,\LLin_{1,3})$) over
\[
\{(f_1,f_2)\in\Der(A,\LLin_{1,2})\times\Der(A,\LLin_{2,3})\mid f_2\otimes f_1\in\Fac^0(A\otimes A,\LLin_{1,3})\}
\]
which in turn is an affine bundle over $\zerocap_{(M_1,M_2,M_3)}$ with fiber
$\Der^0(A,\LLin_{1,2})\times\Der^0(A,\LLin_{2,3})$.

By \cite[Lemma 4.4]{MR1944812}, the first condition on $S$ guarantees that for all $1\le i<j\le 3$,
the dimensions of $\Der(A,\LLin_{i,j})$ and $\Der^0(A,\LLin_{i,j})$ are constant for $(M_1,M_2,M_3)\in S$.

Together with the second condition on $S$, this ensures that the fiber of $\pi$ over any point $(M_1,M_2,M_3)\in S$ is irreducible of constant dimension.
Let $V$ be the vector space
\[
\Lin(A,\LLin_{1,2})\times\Lin(A,\LLin_{2,3})\times\Lin(A,\LLin_{1,3})
\]
(which depends only on $d_1$, $d_2$, $d_3$) with $\Gm$ acting by multiplication by $t$, $t$ and $t^2$ respectively on the three factors.
Then, we can identify $\pi^{-1}(S)$ with a closed $\Gm$-stable subset of
\[
S\times V.
\]
Since $S$ is irreducible, we infer that $\pi^{-1}(S)$ is irreducible by Lemma \ref{lem: irredcrit} below.

Hence, $\Exts_3(S)$ is irreducible since it is the image of $\pi^{-1}(S)$ under the action of $\GL_d$.
\end{proof}

\begin{lemma} \label{lem: irredcrit}
Let $Y$ be an irreducible variety and $V$ a vector space.
Suppose that the multiplicative group $\Gm$ acts trivially on $Y$ and with positive exponents on $V$.
Let $X$ be a closed, $\Gm$-stable subvariety of $Y\times V$ and let $p:X\rightarrow Y$ be the canonical projection.
Suppose that the fibers of $p$ are irreducible and of constant dimension. Then, $X$ is irreducible.
\end{lemma}

\begin{proof}
We may assume that $X\ne Y\times\{0\}$.
It is enough to prove that $X'=X\setminus (Y\times\{0\})$ is irreducible, since it is dense in $X$.
The weighted projective space $\Prj(V)=(V\setminus\{0\})/\Gm$ is a projective variety,
and hence, the map $Y\times\Prj(V)\rightarrow Y$ is closed.
Since $Z:=X'/\Gm$ is a closed subvariety of $Y\times\Prj(V)$,
we obtain a closed map $f:Z\rightarrow Y$ whose fibers are irreducible of constant dimension.
By a standard result (cf.\ proof of \cite[Theorem 1.26]{MR3100243}), $Z$ is irreducible. Therefore, $X'$ is irreducible.
\end{proof}

As always, in the case where $A$ admits a set of orthogonal idempotents such that $e_1+\dots+e_n=1$,
Proposition \ref{prop: irrext3} immediately extends to the module varieties of given graded dimensions.

We will apply this to the case $A=\Pi$.

Harking back to Example \ref{ex: nonasoc}, let $M_1=M_3=S(1)$ and $M_2=S(2)$.
Then, the spaces $\Ext_\Pi^1(M_1,M_2)$, $\Ext_\Pi^1(M_2,M_3)$ and $\Ext_\Pi^2(M_1,M_3)$ are one-dimensional and the cup product \eqref{eq: cuprod}
is non-trivial. Therefore, the quadric $\zerocap_{(M_1,M_2,M_3)}$ is the cross $xy=0$, which is reducible.
This explains the failure of associativity in this case.

In general, we have the following.

\begin{corollary} \label{cor: assoc}
Let $C_i\in\Irrcomp(V^i)$, $i=1,2,3$. Suppose that generically in $C_1\times C_2\times C_3$, the variety $\zerocap_{(M_1,M_2,M_3)}$
is irreducible. Then,
\[
(C_1*C_2)*C_3=C_1*(C_2*C_3).
\]
In particular, this holds if one of the following conditions hold.
\begin{enumerate}
\item $\Ext^1_\Pi(M_1,M_2)\cpp\Ext^1_\Pi(M_2,M_3)=0$ generically in $C_1\times C_2\times C_3$.
For instance, this holds if at least one of the following conditions hold.
\begin{enumerate}
\item $C_1$ and $C_2$ strongly commute.
\item $C_2$ and $C_3$ strongly commute.
\item $\hom_\Pi(C_3,C_1)=0$ (by \eqref{eq: extsrjct}).
\end{enumerate}
\item On an open, nonempty subset of $C_1\times C_2\times C_3$,
$\dim\Ext_\Pi^2(M_1,M_3)=1$ and the rank of the cup product bilinear form \eqref{eq: cuprod} is bigger than one.
\end{enumerate}
\end{corollary}

\begin{proof}
Let $S\subset C_1\times C_2\times C_3$ be an open, nonempty (hence irreducible) subset for which
\begin{gather*}
\dim\Ext_\Pi^1(M_i,M_j)=\ext_\Pi^1(C_i,C_j)\text{ for all }1\le i<j\le 3,\text{ and}\\
\text{the varieties }\zerocap_{(M_1,M_2,M_3)}\text{ are irreducible of constant dimension}
\end{gather*}
for every $(M_1,M_2,M_3)\in S$.
By Proposition \ref{prop: irrext3}, $\Exts_3(S)$ is irreducible.
However, it follows from \eqref{eq: densE} that $\overline{\Exts_3(S)}$ contains both $(C_1*C_2)*C_3$ and $C_1*(C_2*C_3)$,
which are irreducible components of $\Lambda(V^1\oplus V^2\oplus V^3)$.
The corollary follows.
\end{proof}

\begin{remark}
Corollary \ref{cor: assoc} generalizes \cite[Proposition 4.5]{MR3270589}.
\end{remark}

\begin{remark}
Let $U,V,W$ be linear spaces and $B:U\times V\rightarrow W$ a bilinear map. Let
\[
X=\{(u,v)\in U\times V\mid B(u,v)=0\}.
\]
The varieties $\zerocap_{(M_1,M_2,M_3)}$ are of this type.
In general, we do not know a simple exact criterion for the irreducibility of $X$.

For any $u\in U$ let $u^\perp$ be the annihilator of $u$ in $V$ with respect to $B$.
Then, $U^\circ=\argmin_U\dim u^\perp\ne\emptyset$ is open in $U$ \cite[Lemma 4.2]{MR1944812}.
The inverse image of $U^\circ$ under the projection $p_1:X\rightarrow U$ is a vector bundle over $U^\circ$
whose closure $X_1$ is an irreducible component of $X$. Define $X_2$ similarly by interchanging $U$ and $V$.
Obviously, if $X$ is irreducible, then $X_2=X_1$. We do not know whether the converse holds in general.
Note that in general, $X$ may admit irreducible components other than $X_1$ and $X_2$.
In the case where $X=\zerocap_{(M_1,M_2,M_3)}$ one can (ostensibly) weaken the assumption of Corollary \ref{cor: assoc} to assume
that $X_1=X_2$.
\end{remark}

\section{Further results} \label{sec: further}
The following result gives a useful way to construct new rigid modules from old ones.
\begin{lemma} \label{lem: rigidext}
Let
\begin{equation} \label{eq: es}
0\rightarrow x_2\rightarrow x\rightarrow x_1\rightarrow 0
\end{equation}
be a short exact sequence of $\Pi$-modules.
Suppose that $\Ext_\Pi^1(x,x_1)=0$ and $x_2$ is rigid.
Then, $x_1$ and $x$ are rigid.
\end{lemma}

\begin{proof}
Consider the commutative diagram
\[
  \begin{tikzcd}
  \Hom_\Pi(x_2,x)\arrow{r}\arrow{d}{f}  &\Ext_\Pi^1(x_1,x)\arrow{d}  &\\
  \Hom_\Pi(x_2,x_1)\arrow{r}{g}\arrow{d}&\Ext_\Pi^1(x_1,x_1)\arrow{r}&\Ext_\Pi^1(x,x_1)\\
  \Ext_\Pi^1(x_2,x_2)                   &                            &
  \end{tikzcd}
\]
where the middle row and left column are exact. Since $\Ext_\Pi^1(x_1,x)=0$, $g\circ f=0$.
On the other hand, $g$ (resp., $f$) is onto since $\Ext_\Pi^1(x,x_1)=0$ (resp., $\Ext_\Pi^1(x_2,x_2)=0$).
Hence, $\Ext_\Pi^1(x_1,x_1)=0$, so that $x_1$ is rigid. Similarly, the rigidity of $x$ is argued using the diagram
\[
  \begin{tikzcd}
  \Ext_\Pi^1(x_1,x_2)\arrow{r}\arrow{d}&\Ext_\Pi^1(x,x_2)\arrow{r}\arrow{d}&\Ext_\Pi^1(x_2,x_2)\\
  \Ext_\Pi^1(x_1,x)\arrow{r}           &\Ext_\Pi^1(x,x)\arrow{d}           &\\
                                       &\Ext_\Pi^1(x,x_1)                  &
  \end{tikzcd}
\]
The lemma follows.
\end{proof}

Together with Proposition \ref{prop: cancel} we conclude

\begin{corollary} \label{cor: rigidext}
Let $C=C_1*C_2$ with $C_1,C_2\in\Irrcomp$.
Suppose that $C_2$ is rigid and $C$ strongly commutes with $C_1$.
Then $C$ is rigid if and only if $C_1$ is rigid.
\end{corollary}

\begin{remark}
Let $x_1$ and $x_2$ be any $\Pi$-modules. Suppose that $x$ is a rigid $\Pi$-module which is an extension of $x_1$ by $x_2$.
Then, up to isomorphism, $x$ is determined by $x_1$ and $x_2$.
This follows from \cite[Theorem 1.3(ii)]{MR1944812}.
\end{remark}

\begin{remark}
A special case of Lemma \ref{lem: rigidext}, which was proved in \cite[Proposition 5.7]{MR2242628},
is when \eqref{eq: es} does not split and $\dim\Ext_\Pi^1(x_1,x_2)=1$.
Indeed, by \cite[Lemma 5.11]{MR2242628} (which is valid for modules over any algebra $A$) under these conditions,
if $x_1$ is rigid, then $\Ext_A^1(x_1,x)=0$.
For convenience, we recall the argument. By the rigidity of $x_1$, we have an exact sequence
\begin{multline*}
0\rightarrow\Hom_A(x_1,x_2)\rightarrow\Hom_A(x_1,x)\xrightarrow\gamma\Hom_A(x_1,x_1)\xrightarrow\delta
\Ext_A^1(x_1,x_2)\rightarrow\\\Ext_A^1(x_1,x)\rightarrow0.
\end{multline*}
The assumption that \eqref{eq: es} does not split means that $\delta$ is nonzero. Since $\Ext_A^1(x_1,x_2)$ is one-dimensional,
we infer that $\delta$ is surjective. Consequently, $\Ext_A^1(x_1,x)=0$.

A dual argument shows that if $x_2$ is rigid, then $\Ext_A^1(x,x_2)=0$.
Note however that Lemma \ref{lem: rigidext} is more general since we may have $x_1$ and $x_2$ rigid
and $\Ext^1_\Pi(x,x_1)=0$ without $\Ext_\Pi^1(x,x_2)=0$.
\end{remark}

\begin{remark} \label{rem: indct}
Suppose that $x$ is a rigid, non-simple $\Pi$-module.
Then, one may hope that there exists a short exact sequence \eqref{eq: es}
such that $x_1$ and $x_2$ are rigid and non-trivial and $\Ext_\Pi^1(x,x_1)=0$.
If true, this would give an inductive approach to understand rigid modules.
We will prove a special case in Lemma \ref{lem: homtau=0} below.
\end{remark}

We finish this section with another useful result.

\begin{lemma} \label{lem: srigid}
Suppose that $\Hom_\Pi(x_2,x_1)=0$ and $x$ is an extension of $x_1$ by $x_2$.
Let $\gamma\in\Ext^1_\Pi(x_1,x_2)$ be the class of $x$ and let
\[
\omega:\Aut_\Pi(x_1)\times\Aut_\Pi(x_2)\rightarrow\Ext^1_\Pi(x_1,x_2)
\]
be the orbit map pertaining to $\gamma$.
Then,
\[
\dim\Ext_\Pi^1(x,x)=\dim\Ext_\Pi^1(x_1,x_1)+\dim\Ext_\Pi^1(x_2,x_2)+2\dim\Coker d\omega_{e,e}.
\]
In particular, $x$ is rigid if and only if $x_1$ and $x_2$ are rigid and
the $\Aut_\Pi(x_1)\times\Aut_\Pi(x_2)$-orbit of $\gamma$ in $\Ext^1_\Pi(x_1,x_2)$ is open and separable.

Finally, if $x$ is rigid and $x_2$ is a brick (that is, $\End_\Pi(x_2)=K$), then $\Ext_\Pi^1(x_1,x)=0$.
\end{lemma}

\begin{proof}
Consider the commutative diagram
\[
  \begin{tikzcd}
  &\End_\Pi(x_1)\arrow[r,equal]\arrow[d,"\partial_1"]&\Hom_\Pi(x,x_1)\arrow[d]&\\
\End_\Pi(x_2)\arrow[r,"\partial_2"]\arrow[d,equal]  & \Ext_\Pi^1(x_1,x_2)\arrow[r]\arrow[d,"\iota^1"]&
\Ext_\Pi^1(x,x_2)\arrow[r,two heads]\arrow[d,"\beta"]&\Ext_\Pi^1(x_2,x_2)\arrow[d,hookrightarrow]\\
\Hom_\Pi(x_2,x)\arrow[r,"\partial_x"]&\Ext_\Pi^1(x_1,x)\arrow[r,"\alpha"]\arrow[d,two heads]&\Ext_\Pi^1(x,x)\arrow[r,"\beta^*"]
\arrow[d,"\alpha^*"]&\Ext_\Pi^1(x_2,x)\arrow[d]\\
&\Ext_\Pi^1(x_1,x_1)\arrow[r,hookrightarrow]&\Ext_\Pi^1(x,x_1)\arrow[r]&\Ext_\Pi^1(x_2,x_1)
 \end{tikzcd}
\]
with exact rows and columns. Here we used the relation \eqref{eq: extdalty} as well as the assumption
$\Hom_\Pi(x_2,x_1)=0$, which implies that $\Ext_\Pi^2(x_1,x_2)=0$ by \eqref{eq: extsrjct}.
Observe that
\begin{gather*}
\dim\Im\alpha=\dim\Coker\partial_x=\dim\Coker(\iota^1\partial_2)\\=
\dim\Coker\iota^1+\dim\Coker\partial_{12}=\dim\Ext^1_\Pi(x_1,x_1)+\dim\Coker\partial_{12}
\end{gather*}
where
\[
\partial_{12}=(-\partial_1)\oplus\partial_2:\End_\Pi(x_1)\oplus\End_\Pi(x_2)\rightarrow\Ext_\Pi^1(x_1,x_2).
\]
Similarly,
\[
\dim\Im\beta=\dim\Ext^1_\Pi(x_2,x_2)+\dim\Coker\partial_{12}.
\]
Thus,
\begin{gather*}
\dim\Ext_\Pi^1(x,x)=\dim\Im\alpha+\dim\Coker\alpha=\dim\Im\alpha+\dim\Ker\alpha^*\\=
\dim\Im\alpha+\dim\Im\beta\\=\dim\Ext^1_\Pi(x_1,x_1)+\dim\Ext^1_\Pi(x_2,x_2)+2\dim\Coker\partial_{12}.
\end{gather*}
It remains to observe that $\partial_{12}$ is the differential of $\omega$ at the identity.

For the last part, the assumption that $x_2$ is a brick implies that $\Im\partial_2\subset\Im\partial_1$.
Hence, $\partial_x=\iota^1\partial_2=0$ and therefore $\Ext_\Pi^1(x_1,x)=0$ since by assumption, $\alpha=0$.
\end{proof}

\begin{corollary} \label{cor: srigid}
Let $C_1,C_2\in\Irrcomp$ and $C=C_1*C_2$. Suppose that $C$ is rigid and $\hom_\Pi(C_2,C_1)=0$.
Then, $C_1$ and $C_2$ are rigid.
\end{corollary}

\section{Data for \texorpdfstring{$\Pi$}{Pi}-modules\texorpdfstring{\footnote{We take the opportunity to thank
the referee once again for simplifying the discussion of this section.}}{}} \label{sec: taudata}

Let $M$ be a $Q$-representation with defining data $\mu_h:M_{h'}\rightarrow M_{h''}$, $h\in\Omega$.
In order to extend $M$ to a representation of $\Pi$ we need to endow $M$ with a structure
of a $Q^{\op}$-representation (on the same underlying $I$-graded vector space) such the resulting $K\overline{Q}$-module structure
descends to $\Pi$. This data is the same as an element of the kernel of the linear transformation
\[
\oplus_{h\in\Omega}\Lin(M_{h''},M_{h'})\rightarrow\oplus_{i\in I}\Lin(M_i,M_i)
\]
given by
\[
(A_h)_{h\in\Omega}\mapsto(\sum_{h\in\Omega\mid h'=i}A_h\mu_h-\sum_{h\in\Omega\mid h''=i}\mu_hA_h)_{i\in I}.
\]
We can identify this transformation with $\alpha_{\mu,\mu}^*$ (see \eqref{def: alphaMN}).
By \eqref{eq: ESQ} we have $\Ext_Q^1(M,M)^*=\Coker(\alpha_{\mu,\mu})^*=\Ker(\alpha_{\mu,\mu}^*)$,
so we can view a $\Pi$-module as a $Q$-representation $M$ together with an element of $\Ext_Q^1(M,M)^*$.
Alternatively, by Voigt's Lemma this is the same as an element of the conormal of the $G_V$-orbit of $M$ in $R_Q(V)$.
Geometrically, this reflects the fact that the fiber over $0$ of the moment map \eqref{eq: mmap} is the union of the conormal bundles of the $G_V$-orbits
in $R_Q(V)$.
This discussion is closely related to \cite{MR1648647} except that we do not
explicitly use the Auslander--Reiten translation functor. We will comment further about it
at the end of the section.

Next, we describe $\Hom_\Pi$ and $\Ext^1_\Pi$ in this picture.
Since $\Ext(\cdot,\cdot)$ is a bifunctor, for any $Q$-representations $M$ and $N$ we have a canonical bilinear map
\[
\Hom_Q(M,N)\times\Ext^1_Q(N,M)\rightarrow\Ext_Q^1(M,M)\oplus\Ext_Q^1(N,N).
\]
Dually, for any $f\in\Hom_Q(M,N)$ we get a linear map
\[
f^*:\Ext_Q^1(M,M)^*\oplus\Ext_Q^1(N,N)^*\rightarrow\Ext^1_Q(N,M)^*.
\]
Now, if $M$ and $N$ are $\Pi$-modules with data
\[
\dat\in\Ext_Q^1(M,M)^*,\ \dat'\in\Ext_Q^1(N,N)^*,
\]
then we get a linear map
\begin{equation} \label{eq: deflinT}
\linT_{M;N}:\Hom_Q(M,N)\longrightarrow\Ext_Q^1(N,M)^*
\end{equation}
by
\[
\linT_{M;N}f=f^*(\dat,-\dat').
\]

The following is a slight elaboration on \cite[\S8]{MR2360317} which is probably well-known to experts.
For convenience, we include some details.

\begin{proposition} \label{prop: lesdual}
Let $M$ and $N$ be two $\Pi$-modules. Then, we have a functorial long exact sequence
\begin{equation} \label{eq: les22}
\begin{aligned}
0\rightarrow&\Hom_\Pi(M,N)\rightarrow\Hom_Q(M,N)\xrightarrow{\linT_{M;N}}\Ext^1_Q(N,M)^*\rightarrow\\
&\Ext_\Pi^1(M,N)\rightarrow\Ext_Q^1(M,N)\xrightarrow{\linT_{N;M}^*}
\Hom_Q(N,M)^*\rightarrow\\&\Hom_\Pi(N,M)^*\rightarrow0.
\end{aligned}
\end{equation}
In particular,
\begin{equation}
\Hom_\Pi(M,N)=\Ker\linT_{M;N}
\end{equation}
and we have a functorial short exact sequence
\[
0\rightarrow\Coker(\linT_{M;N})\rightarrow\Ext_\Pi^1(M,N)\rightarrow
\Ker(\linT_{N;M}^*)=\Coker(\linT_{N;M})^*\rightarrow0.
\]
\end{proposition}

\begin{proof}
Let $\mu_h:M_{h'}\rightarrow M_{h''}$, $\mu'_h:M_{h''}\rightarrow M_{h'}$
(resp., $\nu_h:N_{h'}\rightarrow N_{h''}$, $\nu'_h:N_{h''}\rightarrow N_{h'}$), $h\in\Omega$ be the data defining $M$
(resp. $N$). Set
\[
X=Z=\oplus_{i\in I}\Lin(M_i,N_i)=\oplus_{i\in I}\Lin(N_i,M_i)^*
\]
and $Y=Y_1\oplus Y_2$ where
\begin{gather*}
Y_1=\oplus_{h\in\Omega}\Lin(M_{h'},N_{h''})=\oplus_{h\in\Omega}\Lin(N_{h''},M_{h'})^*,\\
Y_2=\oplus_{h\in\Omega}\Lin(M_{h''},N_{h'})=\oplus_{h\in\Omega}\Lin(N_{h'},M_{h''})^*.
\end{gather*}
By the description of \cite[\S8]{MR2360317}, the cohomologies of the complex
\[
X\xrightarrow{f}Y\xrightarrow{g}Z
\]
are $\Hom_\Pi(M,N)$, $\Ext^1_\Pi(M,N)$ and $\Hom_\Pi(N,M)^*$ where in the notation of \eqref{def: alphaMN}
$f=\alpha_{\mu,\nu}\oplus\alpha_{\mu',\nu'}$ and $g=\alpha_{\nu',\mu'}^*\oplus\alpha_{\nu,\mu}^*$.
The exactness of \eqref{eq: les22} follows from the lemma below together with \eqref{eq: ESQ}.
\end{proof}

\begin{lemma}
Suppose that in an Abelian category we have a complex
\[
X\xrightarrow{f}Y\xrightarrow{g}Z
\]
and a short exact sequence
\[
0\rightarrow Y_2\rightarrow Y\rightarrow Y_1\rightarrow 0.
\]
Let $f_1$ be the projection of $f$ to $Y_1$ and let $g_2$ be restriction of $g$ to $Y_2$.
Then, we have an exact sequence
\begin{align*}
0\rightarrow\Ker f\rightarrow&\Ker f_1\xrightarrow{f}\Ker g_2\rightarrow\Ker g/\Im f
\rightarrow\\&\Coker f_1\xrightarrow{g}\Coker g_2\rightarrow\Coker g\rightarrow0.
\end{align*}
\end{lemma}

Indeed, this is just the long exact sequence in cohomology pertaining to the short exact sequence of complexes
\[
  \begin{tikzcd}
0\arrow{r}\arrow{d}&0\arrow{r}\arrow{d}&X\arrow[r,equal]\arrow{d}{f}&X\arrow{r}\arrow{d}{f_1}&0\arrow{d}\\
0\arrow{r}\arrow{d}&Y_2\arrow{r}\arrow{d}{g_2}&Y\arrow{r}\arrow{d}{g}&Y_1\arrow{r}\arrow{d}&0\arrow{d}\\
0\arrow{r}&Z\arrow[r,equal]&Z\arrow{r}&0\arrow{r}&0
  \end{tikzcd}
\]

\begin{corollary} \label{cor: dimext}
In the notation above we have
\[
\dim\Ext_\Pi^1(M,N)=\dim\Coker\linT_{M;N}+\dim\Coker\linT_{N;M}.
\]
In particular, $\Ext_\Pi^1(M,N)=0$ if and only if both $\linT_{M;N}$ and $\linT_{N;M}$ are surjective.
\end{corollary}


\begin{corollary} \label{cor: scc}
Suppose that $Q$ is of Dynkin type. Then, for any $\m,\n\in\Mult$,
\begin{multline} \label{eq: C12}
\prm_Q(\m)*\prm_Q(\n)=\prm_Q(\m+\n)\iff\\\linT_{N;M}\text{ is surjective for some }(M,N)\in\cn_Q(\m)\times\cn_Q(\n).
\end{multline}
(The latter is an open condition in $(M,N)$.)
In particular,
\begin{multline} \label{eq: scc}
\text{$\prm_Q(\m)$ and $\prm_Q(\n)$ strongly commute }\iff\\
\prm_Q(\m)*\prm_Q(\n)=\prm_Q(\m+\n)=\prm_Q(\n)*\prm_Q(\m).
\end{multline}
\end{corollary}

\begin{proof}
Let $M$ and $N$ be the $\Pi$-modules with underlying $Q$-representations
$M_Q(\m)$ and $M_Q(\n)$ and data
\[
\dat\in\Ext_Q^1(M_Q(\m),M_Q(\m))^*\text{ and }\dat'\in\Ext_Q^1(M_Q(\n),M_Q(\n))^*.
\]
Then, by Proposition \ref{prop: lesdual}, $\linT_{N;M}$ is surjective if and only if
\begin{equation} \label{eq: id0}
\text{the map $\Ext^1_\Pi(M,N)\longrightarrow\Ext^1_Q(M_Q(\m),M_Q(\n))$ is identically zero.}
\end{equation}
If the latter condition is satisfied for some data $\dat$ and $\dat'$, then since this condition is open,
$\Exts_2(S)\subset\cn_Q(\m+\n)$ for a nonempty open subset $S$ of $\cn_Q(\m)\times\cn_Q(\n)$.
Thus, $\prm_Q(\m)*\prm_Q(\n)\subset\overline{\cn_Q(\m+\n)}=\prm_Q(\m+\n)$, and hence $\prm_Q(\m)*\prm_Q(\n)=\prm_Q(\m+\n)$.

Conversely, if $\prm_Q(\m)*\prm_Q(\n)=\prm_Q(\m+\n)$, then by \eqref{eq: densE}, since $\cn_Q(\m+\n)$ is open in $\prm_Q(\m+\n)$,
there exists $M\in\cn_Q(\m)$ and $N\in\cn_Q(\n)$ such that the set $E$ of extensions of $M$ by $N$ is contained in $\prm_Q(\m+\n)$ and intersects $\cn_Q(\m+\n)$.
Recall that we can write $E$ as $G_V\cdot U$ where $U$ is isomorphic to
a vector space and under this isomorphism, the surjection $\kappa:U\rightarrow\Ext^1_\Pi(M,N)$ is
a quotient map of vector spaces.
Hence, $U\cap\cn_Q(\m+\n)$ is non-empty and open in $U$.
On the other hand, the composition $\tilde\kappa$ of $\kappa$
with $\Ext^1_\Pi(M,N)\rightarrow\Ext^1_Q(M_Q(\m),M_Q(\n))$ is zero
on $U\cap\cn_Q(\m+\n)$. (Here was used the standard fact
that a non-trivial extension of two finite-dimensional modules cannot be isomorphic as a module to the
direct sum.)
Hence $\tilde\kappa\equiv0$ and therefore \eqref{eq: id0} holds. Our assertion follows.
\end{proof}

We also remark that in the Dynkin case we have
\begin{equation} \label{eq: m1m2c}
\prm_Q(\m+\n)\supset\prm_Q(\m)\oplus\prm_Q(\n)
\end{equation}
for any $\m,\n\in\Mult$.

Next, we mention a partial converse to Lemma \ref{lem: rigidext} (cf.\ Remark \ref{rem: indct}).

\begin{lemma} \label{lem: homtau=0}
Let $x$ be a $\Pi$-module with underlying $Q$-representation $M$ and data $\dat\in\Ext_Q^1(M,M)^*$.
Suppose that as a $Q$-representation we have
\[
M=M_1\oplus M_2\text{ with }\Ext_Q^1(M_1,M_2)=0.
\]
Write $\dat=\sm{\dat_1}{}{*}{\dat_2}$ where $\dat_i\in\Ext_Q^1(M_i,M_i)^*$, $i=1,2$
and let $x_i$ be the $\Pi$-module with underlying $Q$-representation $M_i$ and data $\dat_i$.
Then, $x$ is an extension of $x_1$ by $x_2$.
Suppose moreover that $\Hom_Q(M_2,M_1)=0$ and $x$ is rigid.
Then $x_1$ and $x_2$ are rigid. If in addition $\End_Q(M_2)=K$, then $\Ext^1_\Pi(x,x_1)=0$.
\end{lemma}

Indeed, by the vanishing of $\Ext_Q^1(M_1,M_2)$, the projection $M\rightarrow M_1$ lies in the kernel
of $\linT_{x;x_1}$. The first part follows. The other parts follow from Lemma \ref{lem: srigid}.

\subsection*{Relation with Coxeter functor}
Finally, let us assume that $Q$ is acyclic, i.e., $KQ$ is finite-dimensional.
Although the maps $\linT_{M;N}$ defined above admit a concrete linear algebra realization
(taking into account \eqref{eq: ESQ}), at least from a computational aspect it is advantageous to consider
a different point of view using Ringel's formalism \cite{MR1648647}.

Let $\Phi^\pm$ be the Coxeter functors defined by Bernstein--Gel'fand--Ponomarev on the category $\Cat_Q$ of representations of $Q$ \cite{MR0393065}.
Then, $\Phi^-$ is the left adjoint of $\Phi^+$. Let $\epsilon=\epsilon_Q$ be the involution of $KQ$
given by multiplying each arrow by $-1$. It gives rise to an involutive auto-equivalence on $\Cat_Q$ also denoted by $\epsilon$. Let $\tau=\epsilon\Phi^+$. (Of course, $\tau$ is defined for any representation
although we will only consider finite-dimensional ones.)

By the Brenner--Butler--Gabriel theorem (see \cite[Proposition 5.3]{MR607140}), $\tau$ coincides with
the Auslander--Reiten translation functor.
Thus, since $KQ$ is hereditary,
\[
\tau M=\Ext^1_Q(M,KQ)^*.
\]
Here $KQ$ is considered as a left-module over itself; the $\Ext$ space is a right $KQ$-module, and its dual
is a left $KQ$-module.
Also, we have functorial isomorphisms
\[
\Ext_Q^1(M,N)\simeq\Hom_Q(N,\tau M)^*.
\]
Under this isomorphism, a data for a $\Pi$-representation is given by an element $\dat\in\Hom_Q(M,\tau M)$.
If $M$ and $N$ are $\Pi$-modules with data $\dat\in\Hom_Q(M,\tau M)$ and $\dat'\in\Hom_Q(N,\tau N)$, then
\begin{equation}
\linT_{M;N}:\Hom_Q(M,N)\longrightarrow\Hom_Q(M,\tau N)
\end{equation}
is given by
\[
f\mapsto\tau(f)\circ\dat-\dat'\circ f.
\]
In practice, this gives a convenient realization of $\linT_{M;N}$ (for instance,
in translating the rigidity condition of a $\Pi$-module into linear algebra using Corollary \ref{cor: dimext}).

\begin{remark} (\cite{MR545362})
As a $KQ$-module, $\Pi$ is isomorphic to $\oplus_{m\ge0}(\tau^-)^m(KQ)$ where $\tau^-=\Phi^-\epsilon$ is the left adjoint of $\tau$.
Decomposing
\[
KQ=\oplus_{i\in I}KQe_i
\]
as the direct sum of the indecomposable projective $KQ$-modules, we obtain
\begin{equation} \label{eq: sumGP}
\Pi=\oplus_{i\in I}P(i)\text{ where }P(i)=\Pi e_i\simeq\oplus_{m\ge0}(\tau^-)^m(KQe_i).
\end{equation}
If $Q$ is of Dynkin type, then only finitely many summands on the right-hand side are nonzero,
$\Pi$ is finite-dimensional, and the $P(i)$'s are precisely the indecomposable projective $\Pi$-modules.
\end{remark}

\def\cprime{$'$}

\providecommand{\bysame}{\leavevmode\hbox to3em{\hrulefill}\thinspace}
\providecommand{\MR}{\relax\ifhmode\unskip\space\fi MR }
\providecommand{\MRhref}[2]{%
  \href{http://www.ams.org/mathscinet-getitem?mr=#1}{#2}
}
\providecommand{\href}[2]{#2}


\begin{thebibliography}{KKKO18}

\bibitem[BGfP73]{MR0393065}
I.~N. Bern\v{s}te\u{\i}n, I.~M. Gel\cprime~fand, and V.~A. Ponomarev,
  \emph{Coxeter functors, and {G}abriel's theorem}, Uspehi Mat. Nauk
  \textbf{28} (1973), no.~2(170), 19--33. \MR{0393065}

\bibitem[BKT14]{MR3270589}
Pierre Baumann, Joel Kamnitzer, and Peter Tingley, \emph{Affine
  {M}irkovi\'{c}-{V}ilonen polytopes}, Publ. Math. Inst. Hautes \'{E}tudes Sci.
  \textbf{120} (2014), 113--205. \MR{3270589}

\bibitem[CB00]{MR1781930}
William Crawley-Boevey, \emph{On the exceptional fibres of {K}leinian
  singularities}, Amer. J. Math. \textbf{122} (2000), no.~5, 1027--1037.
  \MR{1781930}

\bibitem[CBS02]{MR1944812}
William Crawley-Boevey and Jan Schr\"{o}er, \emph{Irreducible components of
  varieties of modules}, J. Reine Angew. Math. \textbf{553} (2002), 201--220.
  \MR{1944812}

\bibitem[Gab72]{MR0332887}
Peter Gabriel, \emph{Unzerlegbare {D}arstellungen. {I}}, Manuscripta Math.
  \textbf{6} (1972), 71--103; correction, ibid. 6 (1972), 309. \MR{0332887 (48
  \#11212)}

\bibitem[Gab74]{MR0376769}
\bysame, \emph{Finite representation type is open}, Proceedings of the
  {I}nternational {C}onference on {R}epresentations of {A}lgebras ({C}arleton
  {U}niv., {O}ttawa, {O}nt., 1974), {P}aper {N}o. 10 (Ottawa, Ont.), Carleton
  Univ., 1974, pp.~23 pp. Carleton Math. Lecture Notes, No. 9. \MR{0376769 (51
  \#12944)}

\bibitem[Gab80]{MR607140}
\bysame, \emph{Auslander-{R}eiten sequences and representation-finite
  algebras}, Representation theory, {I} ({P}roc. {W}orkshop, {C}arleton
  {U}niv., {O}ttawa, {O}nt., 1979), Lecture Notes in Math., vol. 831, Springer,
  Berlin, 1980, pp.~1--71. \MR{607140 (82i:16030)}

\bibitem[GLS06]{MR2242628}
Christof Gei\ss, Bernard Leclerc, and Jan Schr\"{o}er, \emph{Rigid modules over
  preprojective algebras}, Invent. Math. \textbf{165} (2006), no.~3, 589--632.
  \MR{2242628}

\bibitem[GLS07]{MR2360317}
\bysame, \emph{Semicanonical bases
  and preprojective algebras. {II}. {A} multiplication formula}, Compos. Math.
  \textbf{143} (2007), no.~5, 1313--1334. \MR{2360317}

\bibitem[GLS11]{MR2822235}
\bysame, \emph{Kac-{M}oody groups
  and cluster algebras}, Adv. Math. \textbf{228} (2011), no.~1, 329--433.
  \MR{2822235}

\bibitem[GP79]{MR545362}
I.~M. Gel{\cprime}fand and V.~A. Ponomarev, \emph{Model algebras and
  representations of graphs}, Funktsional. Anal. i Prilozhen. \textbf{13}
  (1979), no.~3, 1--12. \MR{545362 (82a:16030)}

\bibitem[HL10]{MR2682185}
David Hernandez and Bernard Leclerc, \emph{Cluster algebras and quantum affine
  algebras}, Duke Math. J. \textbf{154} (2010), no.~2, 265--341. \MR{2682185}

\bibitem[HL21]{1902.01432}
\bysame, \emph{Quantum affine algebras and cluster algebras}, Interactions of
  Quantum Affine Algebras with Cluster Algebras, Current Algebras and
  Categorification, Progress in Mathematics, vol. 337, Birkh\"auser/Springer,
  Basel, 2021, [In Honor of Vyjayanthi Chari on the Occasion of Her 60th
  Birthday].

\bibitem[Kas91]{MR1115118}
M.~Kashiwara, \emph{On crystal bases of the {$Q$}-analogue of universal
  enveloping algebras}, Duke Math. J. \textbf{63} (1991), no.~2, 465--516.
  \MR{1115118 (93b:17045)}

\bibitem[Kas18]{MR3966729}
\bysame, \emph{Crystal bases and categorifications---{C}hern {M}edal
  lecture}, Proceedings of the {I}nternational {C}ongress of
  {M}athematicians---{R}io de {J}aneiro 2018. {V}ol. {I}. {P}lenary lectures,
  World Sci. Publ., Hackensack, NJ, 2018, pp.~249--258. \MR{3966729}

\bibitem[KKKO15]{MR3314831}
Seok-Jin Kang, Masaki Kashiwara, Myungho Kim, and Se-jin Oh, \emph{Simplicity
  of heads and socles of tensor products}, Compos. Math. \textbf{151} (2015),
  no.~2, 377--396. \MR{3314831}

\bibitem[KKKO18]{MR3758148}
\bysame, \emph{Monoidal categorification of cluster algebras}, J. Amer. Math.
  Soc. \textbf{31} (2018), no.~2, 349--426. \MR{3758148}

\bibitem[KS97]{MR1458969}
Masaki Kashiwara and Yoshihisa Saito, \emph{Geometric construction of crystal
  bases}, Duke Math. J. \textbf{89} (1997), no.~1, 9--36. \MR{1458969
  (99e:17025)}

\bibitem[Lec03]{MR1959765}
B.~Leclerc, \emph{Imaginary vectors in the dual canonical basis of {$U_q(\germ
  n)$}}, Transform. Groups \textbf{8} (2003), no.~1, 95--104. \MR{1959765}

\bibitem[LP97]{MR1428426}
J.~Le~Potier, \emph{Lectures on vector bundles}, Cambridge Studies in Advanced
  Mathematics, vol.~54, Cambridge University Press, Cambridge, 1997, Translated
  by A. Maciocia. \MR{1428426}

\bibitem[Lus90]{MR1182165}
G.~Lusztig, \emph{Canonical bases arising from quantized enveloping algebras.
  {II}}, Progr. Theoret. Phys. Suppl. (1990), no.~102, 175--201 (1991), Common
  trends in mathematics and quantum field theories (Kyoto, 1990). \MR{1182165}

\bibitem[Lus91]{MR1088333}
\bysame, \emph{Quivers, perverse sheaves, and quantized enveloping algebras},
  J. Amer. Math. Soc. \textbf{4} (1991), no.~2, 365--421. \MR{1088333}

\bibitem[Pja75]{MR0390138}
V.~S. Pjasecki{\u\i}, \emph{Linear {L}ie groups that act with a finite number
  of orbits}, Funkcional. Anal. i Prilo\v zen. \textbf{9} (1975), no.~4,
  85--86. \MR{0390138 (52 \#10964)}

\bibitem[Rin98]{MR1648647}
Claus~Michael Ringel, \emph{The preprojective algebra of a quiver}, Algebras
  and modules, {II} ({G}eiranger, 1996), CMS Conf. Proc., vol.~24, Amer. Math.
  Soc., Providence, RI, 1998, pp.~467--480. \MR{1648647}

\bibitem[Sha13]{MR3100243}
Igor~R. Shafarevich, \emph{Basic algebraic geometry. 1}, third ed., Springer,
  Heidelberg, 2013, Varieties in projective space. \MR{3100243}

\bibitem[VV11]{MR2837011}
M.~Varagnolo and E.~Vasserot, \emph{Canonical bases and {KLR}-algebras}, J.
  Reine Angew. Math. \textbf{659} (2011), 67--100. \MR{2837011}

\end{thebibliography}
\end{document}